\documentclass[11pt,a4paper]{article}

\usepackage{amsfonts}
\usepackage{amsthm}
\usepackage{amsmath}
\usepackage{amssymb}
\usepackage{amscd}
\usepackage{mathrsfs}
\usepackage[mathscr]{eucal}
\usepackage{graphicx}
\usepackage{epstopdf} 
\usepackage{pict2e}
\usepackage{epic}
\usepackage{xcolor}
\usepackage{float}
\usepackage{mathtools}
\usepackage{centernot}

\usepackage[margin=1cm, font=small, labelfont=bf]{caption}

\usepackage{cite}
\usepackage{hyperref}
\hypersetup{colorlinks=true, urlcolor= black, linkcolor=black, citecolor=blue}

\usepackage[utf8]{inputenc}
\usepackage[T1]{fontenc}
\usepackage{lmodern}
\usepackage{dsfont}
\usepackage{indentfirst}

\usepackage{hyperref}   
\hypersetup{
    linktoc=page,
    linkcolor=blue,          
    citecolor=blue,        
    filecolor=blue,      
    urlcolor=cyan,
   colorlinks=true           
}

\numberwithin{equation}{section}

\theoremstyle{plain}
\newtheorem{Thm}{Theorem}[section]
\newtheorem{Lem}[Thm]{Lemma}
\newtheorem{Claim}[Thm]{Claim}
\newtheorem{Coro}[Thm]{Corollary}
\newtheorem{Prop}[Thm]{Proposition}

\theoremstyle{definition}
\newtheorem{Def}[Thm]{Definition}

\newtheorem{Rem}[Thm]{Remark}

\newtheorem*{Acknowledgements}{Acknowledgements}
\newtheorem*{Org}{Organisation of the paper}
\newtheorem*{Nota}{Notations}
\newtheorem*{FKG}{The FKG inequality}
\newtheorem*{BK}{The BK inequality}

\usepackage[margin=2.3cm]{geometry}

\definecolor{darkgreen}{rgb}{0,0.6,0.05}

\newcommand{\connect}{\xleftrightarrow}

\title{On reversing the Simon--Lieb inequality in high-dimensional percolation}

\begin{document}

\author{Romain Panis\footnotemark[1]\footnote{Universit\'e Lyon 1, Centrale Lyon, INSA Lyon, Universit\'e Jean Monnet, CNRS, ICJ UMR 5208, 69622, Villeurbanne, France, \url{panis@math.univ-lyon1.fr}, \url{schapira@math.univ-lyon1.fr}}\:, Bruno Schapira\footnotemark[1]}

\date{\vspace{-5ex}}

\maketitle

\abstract{We study Bernoulli percolation on $\mathbb Z^d$ in dimensions ${d>6}$. We prove that a classical consequence of the van den Berg--Kesten inequality, often referred to as the Simon--Lieb inequality in the context of the Ising model, admits a partial reversal. As a main application, we show that the quantity $\varphi_{p_c}(S)$, introduced by Duminil-Copin and Tassion (Comm.\ Math.\ Phys., 2016), is uniformly bounded over all $S\subset \mathbb Z^d$. This partial reversal further yields a short and self-contained route to several key  results, including near-critical estimates on the two-point function and sharp bounds on the critical one-arm probability. 
}


\section{Introduction}\label{sec:intro}

Correlation inequalities play a central role in the study of percolation models. They underpin much of the progress in the field, and the discovery of new ones may well prove decisive in resolving major open problems \cite{kozma2024reduction}. Conversely, settings lacking such tools are notoriously difficult to analyse, see for instance \cite{beffara2025new}. In the context of Bernoulli (independent) percolation, a rich collection of correlation inequalities is available. Among the most fundamental are the \emph{Fortuin--Kasteleyn--Ginibre \textup{(FKG)} inequality} and the \emph{van den Berg--Kesten \textup{(BK)} inequality} \cite{van1985inequalities} (see also \cite{reimer2000proof}). Precise statements are given below, we refer the reader to \cite{GrimmettPercolation1999} for proofs and further background.

The purpose of this work is to study a specific consequence of the BK inequality, commonly referred to as the \emph{Simon--Lieb inequality} by analogy with spin models \cite{SimonInequalityIsing1980,LiebImprovementSimonInequality}. We show that this inequality admits a \emph{partial} reversal. This result is of independent interest, but its full strength emerges in applications, where it provides a new structural tool for the analysis of the mean-field regime of Bernoulli percolation. We state the main results of the paper after introducing some notation and recalling relevant classical results.

\vspace{5pt}

We consider percolation on $\mathbb Z^d$ with $d\geq 2$. We fix a set $\mathcal E$ of edges on $\mathbb Z^d$, i.e.\ a subset of $\{\{x,y\}:x,y\in \mathbb Z^d\}$. If $x,y\in \mathbb Z^d$, we write $x\sim y$ to say that $\{x,y\}\in \mathcal E$.  Given $p\in [0,1]$, we construct a random subgraph of $\mathcal G=(\mathbb Z^d,\mathcal E)$ by independently \emph{opening} (resp. \emph{closing}) each edge of $\mathcal E$ with probability $p$ (resp. $1-p$). The associated measure is denoted by $\mathbb P_p$ (and we write $\mathbb E_p$ for the expectation with respect to $\mathbb P_p$). We consider the following choices for $\mathcal E$:
\begin{enumerate}
	\item[$\bullet$] \emph{Nearest-neighbour} model: $\mathcal E=\{\{x,y\}: \Vert x-y\Vert_1=1\}$ where $\Vert \cdot\Vert_1$ is the $\ell^1$ norm on $\mathbb R^d$;
	\item[$\bullet$] Spread-out model with \emph{spread} parameter $L\geq 1$: $\mathcal E=\{\{x,y\}: \Vert x-y\Vert_1\leq L\}$.
\end{enumerate}
We often analyse the restriction of $\mathbb P_p$ to a subset $S\subset \mathbb Z^d$. By this, we mean that we look at Bernoulli percolation on the graph $(S,E(S))$ where $E(S):=\{\{x,y\}\in \mathcal E: x,y\in S\}$.
 
It is classical (see \cite{GrimmettPercolation1999,DuminilLecturesOnIsingandPottsModels2019}) that the model undergoes a \emph{phase transition} as the parameter $p$ varies:  one has $p_c\in(0,1)$ with 
\begin{equation}
	p_c:=\inf\Big\{p\in [0,1]: \mathbb P_p[0\connect{}\infty]>0\Big\},
\end{equation}
where $\{0\connect{}\infty\}$ is the event that the origin lies in an infinite connected component. 

An important quantity in the analysis of Bernoulli percolation is the so-called \emph{restricted two-point function}. It is defined as follows: if $S\subset \mathbb Z^d$, $p\in (0,1)$, and $x,y\in \mathbb Z^d$,
\begin{equation}
	\tau_{S,p}(x,y):=\mathbb P_p[x\connect{S\:}y],
\end{equation}
where $\{x\connect{S\:}y\}$ is the event that there exists an open path fully contained in $S$ which connects $x$ and $y$. When $S=\mathbb Z^d$, we drop it from the above notation.

We are interested in the emergence, at and near criticality, of \emph{mean-field} properties. This means that we work in sufficiently large dimensions. Mean-field behaviour is predicted to occur in dimensions $d>6$ (and also \emph{marginally} at $d=6$). We refer to \cite{ChayesChayesUpperCritDimPerco1987,GrimmettPercolation1999,SladeSaintFlourLaceExpansion2006,panis2024applications,hutchcroft2025dimension} and references therein for more information on the particular role played by the dimension $d=6$ in percolation theory.

In the case $S=\mathbb Z^d$ and $p=p_c$, it is predicted that, in dimensions $d>6$, for every $x,y\in \mathbb Z^d$,
\begin{equation}\tag{$*$}\label{eq: 2pt full space estimate}
	\tau_{p_c}(x,y)\asymp \frac{1}{(1\vee|x-y|)^{d-2}},
\end{equation}
where $\asymp$ means that the ratio of the two quantities is bounded away from $0$ and infinity by two constants which only depend on $d$ (and potentially on the spread parameter $L$), and where $|\cdot|$ denotes the $\ell^\infty$ norm on $\mathbb R^d$. 

The \emph{lace expansion} approach introduced by Brydges and Spencer \cite{BrydgesSpencerSAW} (see \cite{SladeSaintFlourLaceExpansion2006} for a review) has been successfully implemented to derive a more precise version of \eqref{eq: 2pt full space estimate} for nearest-neighbour percolation in dimensions $d>10$ \cite{HaraSlade1990Perco,HaraDecayOfCorrelationsInVariousModels2008,FitznervdHofstad2017Percod10}, and sufficiently spread-out percolation (i.e. $L\gg 1$) in dimensions $d>6$ \cite{HaraSladevdHofstad2003PercoSO}. A recent alternative proof of \eqref{eq: 2pt full space estimate} in the latter setting was obtained in \cite{DumPan24Perco}. 

As usual in the study of high-dimensional Bernoulli percolation, we work under the assumption that $d>6$ and \eqref{eq: 2pt full space estimate} holds.

\begin{Nota} We let $\Vert \cdot \Vert$ (resp. $|\cdot |$) denote the standard Euclidean norm (resp. the $\ell^\infty$ norm) on $\mathbb R^d$. If $x\in \mathbb R^d$ and $X\subset \mathbb R^d$, we let $\mathrm{d}(x,X):=\inf_{y\in X}|x-y|$. We define $\mathbb S^{d-1}:=\{x\in \mathbb R^d: \Vert x\Vert=1\}$. We denote by $\mathbf{e}_i$ the element of the canonical basis of $i$-th coordinate equal to $1$. If $f,g>0$, we write $f\lesssim g$ (or $g\gtrsim f$) if there exists $C>0$, which only depends on $d$ (and potentially on the spread parameter $L$) such that $f\leq Cg$. If $f\lesssim g$ and $g\lesssim f$, we write $f\asymp g$.

If $S\subset \mathbb Z^d$, we let $\partial S:=\{x\in S: \exists y\notin S, \: x\sim y\}$. We also let $\textup{diam}(S):=\max\{|x-y|:x,y\in S\}$. For $z\in \mathbb Z^d$ and $r\ge 0$, we denote the box of radius $r$ centered at $z$ as 
\begin{equation}
\Lambda_r(z) = \{y\in \mathbb Z^d : |y-z |\le r\},
\end{equation}
and just write $\Lambda_r$ when $z$ is the origin.

Given $A,B,C\subset \mathbb Z^d$, we write $\{A\connect{C\:}B\}$ for the event that there exists an open path in $C$ connecting $A$ and $B$, and we omit the superscript $C$ when $C= \mathbb Z^d$. 

\end{Nota}

\begin{FKG} The FKG inequality states that, for every $p\in [0,1]$ and every \emph{increasing} events (i.e.~events that are stable under the action of opening edges) $E$ and $F$, one has
\begin{equation}\label{eq:FKG}\tag{FKG}
	\mathbb P_p[E\cap F]\geq \mathbb P_p[E]\cdot \mathbb P_p[F].
\end{equation}
\end{FKG}

\begin{BK} If $E$ and $F$ are two percolation events, we write $E\circ F$ for the event of \emph{disjoint} occurrence of $E$ and $F$, i.e.~the event that there exist two disjoint sets $\mathcal I$ and $\mathcal J$ of edges such that the configuration restricted to $\mathcal I$ (resp.~$\mathcal J$) is sufficient to decide that $E$ (resp.~$F$) occurs. The BK inequality states that, for every $p\in [0,1]$ and every increasing events $E$ and $F$, one has
\begin{equation}\label{eq:BK ineq}
\mathbb P_p[E\circ F]\le \mathbb P_p[E]\cdot\mathbb P_p[F].\tag{BK}
\end{equation}
\end{BK}

\subsection{A partially reversed Simon--Lieb inequality}

The central object of this paper is the Simon--Lieb inequality, which provides a way to compare restricted two-point functions involving different sets. 
 
 \begin{Lem}[Simon--Lieb inequality]\label{lem:SL} Let $d\geq 2$. Let $S\subset \Lambda\subset \mathbb Z^d$ and $p\in (0,1)$. For every $o\in S$ and $x\in \Lambda$,
 \begin{equation}
 		\tau_{\Lambda,p}(o,x)\leq \tau_{S,p}(o,x)+\sum_{\substack{u\in S\\v\notin S\\u\sim v}}\tau_{S,p}(o,u) \cdot p \cdot \tau_{\Lambda,p}(v,x).
 \end{equation}
 \end{Lem}
This result is an easy consequence of the BK inequality (see \cite{DuminilTassionNewProofSharpness2016,duminil2017new}). A related inequality (proved by exploration arguments) was derived by Hammersley \cite{hammersley1957percolation}, see the discussion below. We include its proof in Section \ref{sec:SL and its reverse proof} for completeness.
 
 Lemma \ref{lem:SL} plays a central role in the seminal work of Duminil-Copin and Tassion \cite{DuminilTassionNewProofSharpness2016} (see also \cite{SimonInequalityIsing1980} in the context of the Ising model). There, it is used to give an alternative proof of the \emph{sharpness} of the phase transition \cite{Mensikov1986coincidence,AizenmanBarsky1987sharpnessPerco}, which asserts that $p_c$ (as defined above) also corresponds to the threshold beyond which exponential decay of $\tau_p(0,x)$ stops holding. The Simon--Lieb inequality provides a \emph{finite-size criterion} for exponential decay: if there exists a \emph{finite} $S\subset \mathbb Z^d$ containing $0$ such that
 \begin{equation}\label{eq:def phi(S)}
 	\varphi_p(S):=\sum_{\substack{u\in S\\v\notin S\\u\sim v}}\tau_{S,p}(0,u)\cdot p <1,
 \end{equation}
then, iterations of Lemma \ref{lem:SL} imply that exponential decay holds. Remarkably, such an argument can already be found in one of the earliest contributions to percolation theory, due to Hammersley \cite{hammersley1957percolation}. Since the quantity $\varphi_p(S)$ is continuous in $p$, this gives that the set of $p$ for which exponential decay holds is an \emph{open} subset of $[0,1]$. As a consequence, this inequality is not expected to hold for more general percolation models\footnote{For instance, it cannot hold for the free measure of the two-dimensional random-cluster model with cluster weight $q>4$ \cite{duminil2021discontinuity}, since exponential decay of correlations fails to be an open property in that setting.}. The quantity $\varphi_p(S)$ is also interesting on its own, and we return to it below. 
 
In view of the importance of Lemma \ref{lem:SL}, it is natural to ask how sharp this inequality is. Our first result provides a partial answer.

\begin{Thm}[Partially reversed Simon--Lieb inequality]\label{thm:pre reversed SL}  Let $d\geq 2$. Let $S\subset \Lambda\subset \mathbb Z^d$ and $p\in (0,1)$. For every $o\in S$ and $x\in \Lambda$,
\begin{equation}\label{eq:reversed SL with C(u)}
	\tau_{\Lambda,p}(o,x)\geq \tau_{S,p}(o,x)+ \sum_{\substack{u\in S\\v\notin S\\u \sim v}}\mathbb E_p\Big[\mathds{1}\{o\connect{S\:}u\} \cdot p \cdot \tau_{\Lambda\setminus\mathcal C(u;\Lambda),p}(v,x)\Big],
\end{equation}
where $\mathcal C(u;\Lambda)$ denotes the cluster of $u$ in $\Lambda$.
\end{Thm}
To the best of our knowledge, this result is new. We present its (short) proof in Section~\ref{sec:SL and its reverse proof}. 
The terminology ``partially reversed'' refers to the fact that $\tau_{\Lambda\setminus \mathcal C(u;\Lambda),p}$ appears in place of $\tau_{\Lambda,p}$ on the right-hand side of \eqref{eq:reversed SL with C(u)}. To make this bound \emph{effective}, one must understand the impact of restricting the two-point function to the complement of $\mathcal C(u;\Lambda)$. As it turns out, this can be achieved in dimensions $d>6$. The key intuition is that, in this regime, the cluster of $u$ is typically not too large near $u$, so that $\tau_{\Lambda\setminus \mathcal C(u;\Lambda),p}(v,x)$ is comparable to $\tau_{\Lambda,p}(v,x)$. Establishing this directly is generally restricted to the case $\Lambda = \mathbb{Z}^d$ and $p = p_c$ thanks to \eqref{eq: 2pt full space estimate}. Beyond this regime, however, suitably averaged versions of Theorem~\ref{thm:pre reversed SL} allow one to exploit the same intuition. We provide more details in Section \ref{sec:strategy of proof}. 

The following result gives a more effective version of Theorem~\ref{thm:pre reversed SL} in the case $\Lambda = \mathbb{Z}^d$ and $p = p_c$.%
 
\begin{Prop}\label{prop:reversed SL} Let $d>6$ and assume that \eqref{eq: 2pt full space estimate} holds. Let $\varepsilon>0$. There exists $c=c(\varepsilon,d)>0$ such that, for every $S\subset \mathbb Z^d$ finite and every $o\in S$ and $x\in \mathbb Z^d$, 
\begin{equation}\label{eq:reversed SL}
	\tau_{p_c}(o,x)\geq c\sum_{\substack{u\in S: \: |x-u|\geq \varepsilon|x|\\v\notin S\\u\sim v}}\tau_{S,p_c}(o,u)\cdot p_c\cdot\tau_{p_c}(v,x).
\end{equation}
\end{Prop}

Proposition \ref{prop:reversed SL} may not be entirely surprising. Indeed, many correlation inequalities are expected to admit a reversal (at least at the critical point) in the mean-field regime. A more detailed discussion of this phenomenology can be found in \cite[Section~2]{carpenter2025loops}. 

We now turn to applications of Theorem \ref{thm:pre reversed SL}. Our first result concerns the quantity $\varphi_{p_c}(S)$ defined in \eqref{eq:def phi(S)}.

\subsection{Uniform boundedness of $\varphi_{p_c}(S)$}

Understanding $\tau_{S,p_c}$ for a general set $S$ is extremely challenging. The general belief is that in the mean-field regime one should have $\tau_{S,p_c}(0,x)\asymp \mathbb G_S(0,x)$, where $\mathbb G_S$ is the Green function of the simple random walk killed upon exiting $S$. More precisely, letting $\mathbf P_0$ be the law of the simple random walk on $(\mathbb Z^d,\mathcal E)$ started at $0$, we define for $x\in S$,
\begin{equation}
	\mathbb G_S(0,x)=\sum_{k\geq 0}\mathbf P_0[X_k=x, \: T_S>k],
\end{equation}
where $T_S:=\inf\{k\geq 0: X_k\notin S\}$. By \cite{UchiyamaGreenFunctionEstimates1998} and the discussion below 
\eqref{eq: 2pt full space estimate}, this prediction is verified for 
$S = \mathbb{Z}^d$ in certain high-dimensional settings. Under the assumption 
that $d > 6$ and that \eqref{eq: 2pt full space estimate} holds, it was recently 
proved \cite{panis2025sharp} that an analogous result holds for the half-space 
two-point function: setting $\mathbb{H} = \mathbb{N} \times \mathbb{Z}^{d-1}$, 
one has $\tau_{\mathbb{H}, p_c} \asymp \mathbb{G}_{\mathbb{H}}$ (see also 
\cite{ChatterjeeHanson, ChatterjeeHansonSosoe2023subcritical}). Beyond these two cases, no other complete results are available. Even the case $S=\Lambda_n$ is not fully understood (see however \cite[Theorem~2]{ChatterjeeHanson}). We refer to the discussion below \cite[Lemma~1.1]{KozmaNachmias} for more information. 

 Nevertheless, one can study more averaged versions of these restricted two-point functions. A very natural candidate is the quantity $\varphi_{p_c}(S)$ defined in \eqref{eq:def phi(S)}. This quantity is comparable to the expected number of \emph{pioneers}. Recalling that $\partial S = \{x \in S : \exists\, y \notin S,\, x \sim y\}$, we call a point $u \in \partial S$ a \emph{pioneer} if $0 \connect{S\:} u$. We also define $\mathcal{P}_S := |\{u \in \partial S : 0 \connect{S\:} u\}|$. One can check that, for $p\in (0,1)$,
\begin{equation}\label{eq:comparison pioneers and phi}
	(|\{x: 0\sim x\}|p)^{-1}\varphi_{p}(S)\leq \mathbb E_{p}[\mathcal{P}_S]=\sum_{u\in \partial S}\tau_{S,p}(0,u)\leq p^{-1}\varphi_{p}(S).
\end{equation}
In the case of the simple random walk, it is an easy consequence of Markov's property that
\begin{equation}
	\varphi^{\textup{RW}}(S):=\sum_{\substack{u\in S\\v\notin S\\u\sim v}}\mathbb G_S(0,u)\mathbf P_0[X_1=v-u]=\mathbf P_0[T_S<\infty]\leq 1.
\end{equation}
In particular, the $\varphi^{\textup{RW}}(S)$'s, where $S$ runs over subsets of $\mathbb Z^d$ containing $0$, are uniformly bounded. Our main result extends this fact to high-dimensional Bernoulli percolation.

\begin{Thm}\label{thm:main} Let $d>6$ and assume that \eqref{eq: 2pt full space estimate} holds. There exists $C>0$ such that, for every $S\subset \mathbb Z^d$ containing $0$,
\begin{equation}
	\varphi_{p_c}(S)\leq C.
\end{equation}
\end{Thm}
Note that it is not even clear that $\varphi_{p_c}(S) < \infty$ when $S$ is 
infinite. Prior to this work, the only results in this direction were obtained 
for $S =\mathbb H_n:= \mathbb{H} - (n, 0, \ldots, 0)$ in 
\cite{HutchcroftMichtaSladePercolationTorusPlateau2023, DumPan24Perco} (see 
also \cite{panis2025sharp}). Theorem~\ref{thm:main} may come as a 
surprise, as it extends the comparison to the random walk to settings 
where the boundary of $S$ is rough. Let us also mention that the result is 
\emph{optimal}: we expect $\varphi_{p_c}(S)$ to take arbitrarily large values 
in dimensions $2 \leq d \leq 6$. We explain how to derive Theorem~\ref{thm:main} 
from Theorem~\ref{thm:pre reversed SL} in Section~\ref{sec:strategy of proof}.

One of the main difficulties in the analysis of $\varphi_p(S)$ lies in its lack of monotonicity in $S$. Nevertheless, Theorem \ref{thm:main} provides the following result.

\begin{Coro}\label{cor:partial monot} Let $d>6$ and assume that \eqref{eq: 2pt full space estimate} holds. There exists $C>0$ such that, for every $p\leq p_c$, every $S\subset \mathbb Z^d$ containing $0$, and every $\Lambda\supset S$,
\begin{equation}
	\varphi_p(\Lambda)\leq C\cdot \varphi_p(S).
\end{equation}
\end{Coro}

While Theorem~\ref{thm:main} is of independent interest, it also plays a key 
role in the following applications of Theorem~\ref{thm:pre reversed SL}.

\subsection{Near-critical bounds}

Various natural length scales can be defined in the subcritical regime. We focus on the so-called \emph{sharp length} which was initially introduced in \cite{hutchcroft2022derivation,PanisTriviality2023}, and more recently studied in various contexts of statistical mechanics \cite{DumPan24Perco,DumPan24WSAW,DuminilPanis2024newLB,vEGPS,vEGPSperco}. The sharp length $L(p)$ is the first scale $k$ at which we can find a set $S\subset \Lambda_k$ containing $0$ for which $\varphi_p(S)<\tfrac{1}{2}$. More precisely,
\begin{equation}
	L(p):=\inf\Big\{k\geq 1: \exists S\subset \Lambda_k, \: 0\in S, \: \varphi_p(S)<\tfrac{1}{2}\Big\}.
\end{equation}
It is not hard to see that $L(p)$ diverges as $p$ approaches $p_c$ (see for instance \cite{PanisTriviality2023}).
Our next result quantifies this divergence.
\begin{Thm}\label{thm:sharplength} Let $d>6$ and assume that \eqref{eq: 2pt full space estimate} holds. There exists $c,C>0$ such that, for every $p\in [p_c/2,p_c)$,
\begin{equation}
	c(p_c-p)^{-1/2}\leq L(p)\leq C(p_c-p)^{-1/2}.
\end{equation}
\end{Thm}
Combining Theorems \ref{thm:pre reversed SL}, \ref{thm:main}, and \ref{thm:sharplength}, we obtain the following near-critical estimate on $\varphi_p(S)$ for a general set $S$. 

\begin{Thm}\label{thm:nc bounds phi(S)} Let $d>6$ and assume that \eqref{eq: 2pt full space estimate} holds. There exist $c,C>0$ such that, for every $p< p_c$ sufficiently close to $p_c$, and every $S\subset \mathbb Z^d$ containing $0$,
\begin{equation}
	c\exp\Big(-C(p_c-p)^{1/2}\mathrm{diam}(S)\Big)\leq\varphi_p(S)\leq C\exp\Big(-c(p_c-p)^{1/2}\mathrm{d}(0,\partial S)\Big),
\end{equation}
with the convention that $\textup{diam}(S)=\infty$ when $S$ is infinite.
\end{Thm}
Theorem \ref{thm:nc bounds phi(S)} is mostly useful in settings where $\textup{diam}(S)\asymp \mathrm{d}(0,\partial S)$, as is the case for 
instance when $S = \Lambda_n$, $n \geq 1$. Let us mention that this result also implies a lower bound on $\varphi_p(\mathbb H_n)$: indeed, by symmetry, one has $\varphi_p(\mathbb H_n)\geq \tfrac{1}{2d} \cdot\varphi_p(\Lambda_n)$.

Building on Theorem \ref{thm:nc bounds phi(S)}, we also obtain the following near-critical bounds on the two-point function in the full-space.

\begin{Thm}\label{thm:bounds 2pt} Let $d>6$ and assume that \eqref{eq: 2pt full space estimate} holds. There exist $c,C>0$ such that, for every $p< p_c$ sufficiently close to $p_c$, and every $x\in \mathbb Z^d$,
\begin{equation}
	\frac{c}{(1\vee |x|)^{d-2}}\exp\Big(-C(p_c-p)^{1/2}|x|\Big)\leq \tau_p(0,x)\leq \frac{C}{(1\vee |x|)^{d-2}}\exp\Big(-c(p_c-p)^{1/2}|x|\Big).
\end{equation}
\end{Thm}
As mentioned above, there are other natural length-scales to consider in the subcritical regime. The most commonly studied one is the so-called \emph{correlation length}, which measures the inverse of the exponential rate of decay of the two-point function. It is defined, for $p<p_c$ and $u\in \mathbb S^{d-1}$, by
\begin{equation}
	\xi(\mathbf{u},p):=\left(-\frac{1}{n}\lim_{n\rightarrow \infty}\log \tau_p(0,\lfloor n\mathbf{u}\rfloor )\right)^{-1}.
\end{equation}
Another natural distance is given by the \emph{correlation length of order $\phi$} where $\phi>0$: for $p<p_c$,
\begin{equation}
	\xi_\phi(p)^\phi:=\frac{1}{\chi(p)}\sum_{x\in \mathbb Z^d}|x|^\phi\tau_p(0,x),
\end{equation}
where $\chi(p):=\mathbb E_p[|\mathcal C(0)|]=\sum_{x\in \mathbb Z^d}\tau_p(0,x)$ is the expected size of the cluster of the origin $\mathcal C(0)$ (also called the \emph{susceptibility}). 
Theorem \ref{thm:bounds 2pt} immediately implies the following result.

\begin{Coro}\label{cor:correlation length} Let $d>6$ and assume that \eqref{eq: 2pt full space estimate} holds. There exist $c,C>0$ such that, for every $p< p_c$ sufficiently close to $p_c$, for every $\mathbf{u}\in \mathbb S^{d-1}$,
\begin{equation}
	c (p_c-p)^{-1/2}\leq \xi(\mathbf{u},p)\leq C (p_c-p)^{-1/2},
\end{equation}
and for every $\phi>0$,
\begin{equation}
		c (p_c-p)^{-1/2}\leq \xi_{\phi}(p)\leq C (p_c-p)^{-1/2}.
\end{equation}
\end{Coro}

Except for the lower bounds of Theorems \ref{thm:nc bounds phi(S)} and \ref{thm:bounds 2pt}, the above results were already known. The main novelty is in their proofs, which are very different from previous approaches and also much shorter. In particular, we only rely on classical results of Aizenman--Newman \cite{AizenmanNewmanTreeGraphInequalities1984} and Aizenman \cite{AizenmanNumberIncipient1997} to derive them.

Although not explicitly stated there, Theorem \ref{thm:sharplength} follows from the near-critical analysis of \cite{HutchcroftMichtaSladePercolationTorusPlateau2023}. It was also derived in \cite{DumPan24Perco} in the context of spread-out percolation. A weaker version of Theorem \ref{thm:nc bounds phi(S)} (restricted to $S=\mathbb H_n$ for $n\geq 0$), as well as the upper bound of Theorem \ref{thm:bounds 2pt} are also proved in \cite{HutchcroftMichtaSladePercolationTorusPlateau2023}. However, \cite{HutchcroftMichtaSladePercolationTorusPlateau2023} relies crucially on the one-arm computation of \cite{KozmaNachmias} and on the estimate on the half-space two-point function of \cite{ChatterjeeHanson} to conclude. Theorems \ref{thm:pre reversed SL} and \ref{thm:main} allow us to circumvent the use of these results. Corollary \ref{cor:correlation length} was obtained for $\mathbf{u}=(1,0,\ldots,0)$ or $\phi=2$ in \cite{hara1990mean} using the lace expansion.
Finally, let us mention that much more precise versions of Theorem \ref{thm:bounds 2pt} and Corollary \ref{cor:correlation length} were recently derived for nearest-neighbour percolation in dimensions $d\geq 15$ \cite{LiuSla26}.

\subsection{The one-arm exponent}

We now turn to our last application, beginning with some definitions. If $n\geq 1$ and $p\in [0,1]$, we define the \emph{one-arm probability} by
\begin{equation}
	\theta_n(p):=\mathbb P_p[0\connect{}\partial\Lambda_n].
\end{equation}
It is expected that $\theta_n(p_c)$ decays algebraically in $n$. The exponent governing the decay is called the \emph{one-arm exponent}. In \cite{duminil2017new}, the authors proved that
\begin{equation}\label{eq:derivative of theta}
	\frac{\mathrm{d}}{\mathrm{d}p}\theta_n(p)=\frac{1}{p(1-p)}\mathbb E_p[\varphi_p(\mathcal S_n)\mathds{1}_{0\in \mathcal S_n}],
\end{equation}
where $\mathcal S_n:=\{x\in \Lambda_n: x\textup{ is not connected to } \partial\Lambda_n \}$. Combining \eqref{eq:derivative of theta} with Theorem \ref{thm:main} yields the following result.
\begin{Coro}\label{coro:bounded derivative} Let $d>6$ and assume that \eqref{eq: 2pt full space estimate} holds. There exists $C>0$ such that, for every $n\geq 1$ and every $p\in [p_c/2,p_c]$,
\begin{equation}\label{eq:derivative bounded}
	\frac{\mathrm{d}}{\mathrm{d}p}\theta_n(p)\leq C.
\end{equation}
\end{Coro}
Combined with Theorem \ref{thm:nc bounds phi(S)}, Corollary \ref{coro:bounded derivative} yields a (very) short route to the computation of the one-arm exponent, thereby reproving a celebrated result of Kozma and Nachmias \cite{KozmaNachmias}.

\begin{Thm}\label{thm:onearm} Let $d>6$ and assume that \eqref{eq: 2pt full space estimate} holds. There exist $c,C>0$ such that, for every $n\geq 1$,
\begin{equation}\label{eq:one-arm proba}
	\frac{c}{n^2}\leq \theta_n(p_c)\leq \frac{C}{n^2}.
\end{equation}\end{Thm}

The argument of Kozma and Nachmias \cite{KozmaNachmias} relies on the analysis of \emph{regular points}. We will introduce them more precisely in Section \ref{sec:strategy of proof} below. Roughly, a point $u$ is regular if its cluster is not too dense around it. The definition in \cite{KozmaNachmias} is a bit more involved, which makes their analysis quite intricate. As we glimpsed above, this notion is relevant to make Theorem \ref{thm:pre reversed SL} effective. We use a more geometric notion of regular points (introduced in \cite{Asselah2025capacity}) to prove Theorem \ref{thm:main}. More details are provided in the next subsection.

Recently, two alternative proofs of Theorem~\ref{thm:onearm} have appeared. To place our work in context, we briefly review these approaches. In \cite{Asselah2025capacity}, the authors introduce a simpler notion of regular points, which simplifies part of the argument of Kozma and Nachmias. However, the overall strategy remains unchanged, and technical difficulties limit the approach to dimensions $d>7$. A different method is developed in \cite{vEGPSperco}, where the authors recover \eqref{eq:one-arm proba} using the so-called \emph{entropic bound} pioneered by Dewan and Muirhead \cite{DewanMuirhead} (see also \cite{hutchcroft2022derivation}). With this tool, the computation of the one-arm exponent reduces to showing that there exists $K$ large enough such that, for every sufficiently large $n$, 
\begin{equation}\label{eq:theta drops intro}
	\theta_n(p_c-K/n^2)\leq \frac{1}{10}\theta_{n/2}(p_c).
\end{equation}
As we explain below (see Lemma \ref{lem:theta drops}), this bound follows readily from the sharp length estimate $L(p)\asymp (p_c-p)^{-1/2}$. Consequently, \cite{vEGPSperco} obtains \eqref{eq:one-arm proba} in settings where this input can be established independently of \cite{KozmaNachmias}. This is the case for sufficiently spread-out percolation in dimensions ${d>6}$ \cite{DumPan24Perco}, but not the case in the current more general setting (indeed \cite{HutchcroftMichtaSladePercolationTorusPlateau2023} relies on \cite{KozmaNachmias}). 
In contrast, our approach provides a short and essentially self-contained proof of the sharp length estimate, but also bypasses the use of the entropic bound.

\subsection{Strategy of proofs}\label{sec:strategy of proof}
We now briefly discuss the main ideas of this paper. The novelties lie in the proofs of Theorems \ref{thm:pre reversed SL}, \ref{thm:main}, \ref{thm:sharplength}, and in the lower bounds of Theorems \ref{thm:nc bounds phi(S)} and \ref{thm:bounds 2pt} so we focus on these. We assume that $d>6$ and that \eqref{eq: 2pt full space estimate} holds.
\vspace{5pt}

Let us first discuss how to obtain bounds on $\varphi_{p_c}(S)$ for a finite $S\subset \mathbb Z^d$ containing $0$. It is quite easy to prove a lower bound on $\varphi_{p_c}(S)$. Indeed, by Lemma \ref{lem:SL}, if $x\notin S$, \begin{equation}\label{eq:SL}
	\tau_{p_c}(0,x)\leq \sum_{\substack{u\in S\\ v\notin S\\ u\sim v}}\tau_{S,p_c}(0,u)\cdot p_c\cdot \tau_{p_c}(v,x).
\end{equation}
If $|x|\geq 2\cdot\textup{diam}(S)$, \eqref{eq: 2pt full space estimate} gives that, for every $v$ as above, $\tau_{p_c}(v,x)\asymp \tau_{p_c}(0,x)$, so that
\begin{equation}
	\varphi_{p_c}(S)\gtrsim 1.
\end{equation}
Let us stress, however, that obtaining this bound does not require \eqref{eq: 2pt full space estimate}: the bound $\varphi_{p_c}(S)\geq 1$ holds in any dimension $d\geq 2$, see \cite{DuminilTassionNewProofSharpness2016}. 

It is thus rather natural to try to obtain the inequality $\varphi_{p_c}(S)\lesssim 1$ by deriving a ``reversed'' Simon--Lieb inequality. This is exactly what is performed in Proposition \ref{prop:reversed SL}. Using again \eqref{eq: 2pt full space estimate}, we see that Theorem \ref{thm:main} and Proposition \ref{prop:reversed SL} are equivalent. We will in fact obtain Proposition \ref{prop:reversed SL} as a corollary of Theorem \ref{thm:main}. Our purpose is to advertise that a good strategy to derive our main results involves trying to show that, in high dimensions, the application of the BK inequality to connectivity events is \emph{sharp} (up to a multiplicative constant). Such an idea is also at the core of our recent paper \cite{panis2025sharp}.

The strategy of reversing \eqref{eq:SL} to upper bound $\varphi_{p_c}(\mathbb H_n)$, where we recall that $\mathbb H_n=\mathbb H-(n,0,\ldots,0)$, was pursued in \cite{DumPan24Perco} (see also \cite{DumPan24WSAW} in the context of weakly self-avoiding walks). However, \eqref{eq:reversed SL} is within a \emph{multiplicative} constant from the upper bound of \eqref{eq:SL}, while in \cite{DumPan24Perco} it is within an \emph{additive} ``error term''. Estimating the latter for a general $S$ is extremely difficult as it involves a priori estimates on $\tau_{S,p_c}$ (which are usually not available). Let us also mention that \cite{HutchcroftMichtaSladePercolationTorusPlateau2023} obtains an upper bound on $\varphi_{p_c}(\mathbb H_n)$ by proving the existence of $c,n_0>0$ such that, for every $\ell,k\geq 0$, 
\begin{equation}
\varphi_{p_c}(\mathbb H_{k+\ell+n_0})\geq c\cdot \varphi_{p_c}(\mathbb H_k)\cdot \varphi_{p_c}(\mathbb H_{\ell}),
\end{equation} 
see \cite[Proposition~3.2]{HutchcroftMichtaSladePercolationTorusPlateau2023}.
Such an inequality can be viewed as an averaged reversed Simon--Lieb inequality. However, their analysis uses translation invariance and requires bounds on $\tau_{\mathbb H,p_c}$, and hence cannot be used to bound $\varphi_{p_c}(S)$ for a general set $S$. 

We now discuss the main intuition behind the derivation of Theorems \ref{thm:pre reversed SL} and \ref{thm:main}. We start by looking at the counterpart of Lemma \ref{lem:SL} in the context of random walks. Exploring the random walk until it exits $S$ for the first time, and applying Markov's (strong) property, yields, for every $x\notin S$,
\begin{equation}
	\mathbb G(0,x)=\sum_{\substack{u\in S\\v\notin S\\u\sim v}}\mathbb G_S(0,u)\cdot \mathbf P_0[X_1=v-u]\cdot \mathbb G(v,x).
\end{equation}
In the setting of Bernoulli percolation, the situation is considerably more delicate. Two main obstacles prevent a direct adaptation of the random-walk argument: 
\begin{enumerate}
	\item[$(a)$] Since we are considering a percolation model rather than a path model, the ``first edge exiting $S$'' is not necessarily well-defined.
	\item[$(b)$] Ignoring $(a)$ and letting $\{u,v\}$ be such that $0\connect{S\:}u$ and $v\connect{}x$, we observe that the two connections are not independent. Indeed, it is possible that the connection between $v$ and $x$ visits the cluster that has been explored to reveal the connection between $0$ and $u$ in $S$. This difficulty arises from the non-Markovian nature of the cluster exploration.
\end{enumerate}
We solve these problems as follows. 

First, to mark a \emph{special} edge on the boundary of $S$, we observe that we may create a connection from $0$ to $x$ by asking, for an edge $\{u,v\}$ with $u \in S$ and $v\notin S$, that the following event occurs: 
\begin{equation}\label{eq:def E}
\mathcal E_{uv}:=\{0\connect{S\:}u\}\cap \{\{u,v\} \textup{ is open}\}\cap \{v\connect{}x\}\cap \{0\connect{\omega^{[uv]}\:}x\}^c,
\end{equation}
where $\omega^{[uv]}$ is the percolation configuration obtained from $\omega$ by closing the edge $\{u,v\}$. In words, this event means that the edge $\{u,v\}$ is open and pivotal for $\{0\connect{}x\}$, \emph{and} that $0\connect{S\:}u$. We observe that a very similar event appears in the analysis of the susceptibility of the model \cite{AizenmanNewmanTreeGraphInequalities1984}. As we prove below, the events $\mathcal E_{uv}$ are incompatible, thus providing a unique way to mark an edge $\{u,v\}$ on the boundary of $S$. At this stage, we may be losing a lot since this construction may not capture the true structure of a connection from $0$ to $x$. Nevertheless, this is enough to get Theorem \ref{thm:pre reversed SL}, and we will see retrospectively that such a picture is true in the high-dimensional setting.

Second, to solve $(b)$, we rely on an extra ingredient due to Aizenman \cite{AizenmanNumberIncipient1997}: 
in dimensions $d>6$, it holds (with high probability) that $|\mathcal C(u)\cap \Lambda_n(u)|\lesssim n^{4+o(1)}$ for every $n$ large enough (where $o(1)$ tends to $0$ as $n$ tends to infinity). Let $A$ be a sample of $\mathcal C(u)$ satisfying these bounds. Then, by \eqref{eq: 2pt full space estimate},
\begin{equation}
\mathbb E_{p_c}\big[|\mathcal C(v)\cap A|\big]\lesssim\sum_{k\geq 0}\frac{(2^k)^{4+o(1)}}{2^{k(d-2)}}\lesssim 1,
\end{equation}
where we used that $d>6$ in the last inequality. This computation suggests (but does not prove) that, conditionally on its existence, the connection from $v$ to $x$ has a positive probability to occur in the complement of $\mathcal C(u)$. See Figure \ref{fig:reversed SL} for an illustration. 

\begin{figure}[htb]
	\begin{center}
	\includegraphics{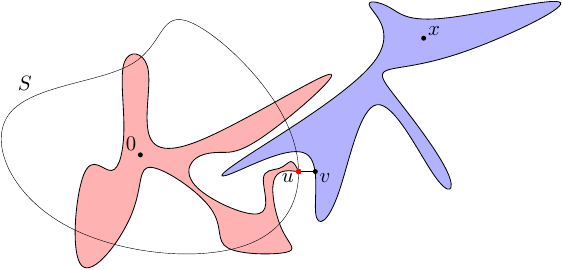}
	\caption{An illustration of the event $\mathcal E_{uv}$ introduced in \eqref{eq:def E}. The cluster of $u$ (resp. $v$) in $\omega^{[uv]}$ is the light red (resp. blue) region. Our entire strategy revolves around the observation that, if $u$ is regular and $d>6$, then the blue and red regions have a positive probability to avoid each other. 
	}
	\label{fig:reversed SL}
	\end{center}
\end{figure}

The points $u\in \partial S$ which satisfy the above property are called \emph{regular points}. They were first introduced by Kozma and Nachmias \cite{KozmaNachmias} in the context of arm exponents computations. Here, we rely on a more recent (and easier) perspective on this notion that was used in \cite{Asselah2025capacity,panis2025sharp}. A precise definition is given in Section \ref{sec:reg,escap,effective}.  

To prove Theorem \ref{thm:main}, we first establish (see Proposition \ref{prop: effective reversed SL}) a modification of Theorem \ref{thm:pre reversed SL} in which we further ask $u$ to be regular (in fact we will also require $u$ to be \emph{escapable}, see Definition \ref{def: escap}). We view this result as an effective version of Theorem \ref{thm:pre reversed SL}, in the sense that it is well-suited for the study of high-dimensional percolation. Combining this with \eqref{eq: 2pt full space estimate}, and letting $\mathcal{P}_S^{\textup{reg}}:=|\{u\in \partial S: 0\connect{S\:}u, \: u\textup{ is regular}\}|$, we obtain
\begin{equation}\label{eq:stepA}
	\mathbb E_{p_c}[\mathcal{P}_S^{\textup{reg}}]\lesssim 1,
\end{equation}
where the implicit constant does not depend on $S$.

 The second part of the argument is completely independent, and follows  the strategy used in our previous paper \cite[Lemma~4.3]{panis2025sharp}. Recalling that $\mathcal P_S=|\{u\in \partial S: 0\connect{S\:}u\}|$ denotes the number of pioneers of $S$, we prove that
\begin{equation}\label{eq:stepB}
	\mathbb E_{p_c}[\mathcal P_S]\lesssim \mathbb E_{p_c}[\mathcal P_S^{\textup{reg}}],
\end{equation}
see Lemma \ref{lem: pioneers are in average regular} for a precise statement.
In words, we show that, on average, pioneers are regular.
The combination of \eqref{eq:stepA} and \eqref{eq:stepB} readily gives Theorem \ref{thm:main} by \eqref{eq:comparison pioneers and phi}.

\vspace{5pt}

The above analysis relies on \eqref{eq: 2pt full space estimate}, which provides 
a complete understanding of $\tau_{p_c}$. In the regime $p < p_c$, however, no 
lower bound on $\tau_p$ is available. The second set of innovations of this paper 
concerns the analysis of Theorem~\ref{thm:pre reversed SL} in this more 
challenging regime. We begin by describing how Theorem~\ref{thm:sharplength} is 
obtained. As in \cite{DumPan24Perco}, the idea is to build on the following result of Aizenman and Newman \cite{AizenmanNewmanTreeGraphInequalities1984}: if $d>6$ and \eqref{eq: 2pt full space estimate} holds, then 
\begin{equation}\label{eq:susceptibility}
	\chi(p)=\sum_{x\in \mathbb Z^d}\tau_p(0,x)\asymp (p_c-p)^{-1}.
\end{equation}
We will prove that $\chi(p) \asymp L(p)^2$, from which Theorem~\ref{thm:sharplength} 
follows immediately. The main difficulty lies in establishing the lower bound 
$\chi(p) \gtrsim L(p)^2$. This result would be relatively easy if one had a version of Theorem \ref{thm:pre reversed SL} without the restriction to $\mathcal C(u)^c$, see Figure \ref{fig:proof L(p)}. As it turns out, working with the averaged quantity $\chi(p)$ allows 
us to use the strategy of Theorem \ref{thm:main} to handle this restriction without using any pointwise lower bound on $\tau_p$. The proofs of the lower bounds in Theorems~\ref{thm:nc bounds phi(S)} 
and~\ref{thm:bounds 2pt} follow a similar but technically more involved 
argument. The additional difficulty being that, unlike before, we now work in a regime 
where $\varphi_p(S)$ may be very small.
\vspace{5pt}

We conclude the introduction by mentioning that the techniques developed in this 
paper are, we believe, robust enough to be used beyond the setting of Bernoulli 
percolation. In a forthcoming work~\cite{PanIsing}, they will be applied to study the restricted two-point functions of the Ising model in dimensions $d > 4$.

\begin{Org} In Section \ref{sec:SL and its reverse proof}, we give the (short) proofs of Lemma \ref{lem:SL} and Theorem \ref{thm:pre reversed SL}. Section \ref{sec:reg,escap,effective} gives a precise description of the new tool for high-dimensional percolation presented in this paper. In Section \ref{sec: subsec reg esc}, we introduce the key concepts which help to use Theorem \ref{thm:pre reversed SL} in the high-dimensional setting. The effective version of Theorem \ref{thm:pre reversed SL} that is used throughout is presented there, in Proposition \ref{prop: effective reversed SL}. In Section \ref{sec: subsec average}, we prove that, on average,  pioneers are regular (and a little bit more), see Proposition \ref{prop: in average pioneers are regular and escapable}. Therefore, the restrictions of Proposition \ref{prop: effective reversed SL} are (on average) harmless. In Section \ref{sec:proof bound phi(S)}, we prove Theorem \ref{thm:main} as well as two easy consequences of it: Proposition \ref{prop:reversed SL} and Corollary \ref{cor:partial monot}. In Section \ref{sec:sharp length}, we prove Theorem \ref{thm:sharplength} and immediately apply it to derive the upper bounds in Theorems \ref{thm:nc bounds phi(S)} and \ref{thm:bounds 2pt}. The short proof of Theorem \ref{thm:onearm} is carried out in Section \ref{sec:one arm}. Section \ref{sec:nc lb} is devoted to the more demanding proofs of the lower bounds in Theorems \ref{thm:nc bounds phi(S)} and \ref{thm:bounds 2pt}.
 
\end{Org}

\begin{Acknowledgements} RP thanks Hugo Duminil-Copin, Diederik van Engelenburg, Christophe Garban, and Franco Severo for stimulating discussions regarding the problem. We thank Geoffrey Grimmett, Aman Markar, and Akira Sakai for useful comments. RP acknowledges the support of the Swiss National Science Foundation through a Postdoc.Mobility grant. BS acknowledges support from the grant ANR-22-CE40-0012 (project LOCAL).
\end{Acknowledgements}

\section{Proof of the partially reversed Simon--Lieb inequality}\label{sec:SL and its reverse proof}

In this section, we prove Lemma \ref{lem:SL} and Theorem \ref{thm:pre reversed SL}. We let $d\geq 2$ and fix $p\in (0,1)$.

\begin{proof}[Proof of Lemma~\textup{\ref{lem:SL}}] We fix $S\subset \Lambda \subset \mathbb Z^d$. Let $o\in S$ and $x \in \Lambda$. Observe that
\begin{equation}
	\{o\connect{\Lambda\:}x\}\setminus \{o\connect{S\:}x\}\subset \bigcup_{\substack{u\in S\\ v\notin S\\ u\sim v}}\{o\connect{S\:}u\}\circ \{\{u,v\} \textup{ is open}\}\circ \{v\connect{\Lambda\:}x\}.
\end{equation}
Using the BK inequality \eqref{eq:BK ineq}, we obtain
\begin{equation}\label{eq:SL proof}
	\tau_{\Lambda,p}(o,x)-\tau_{S,p}(o,x)\leq \sum_{\substack{u\in S\\ v\notin S\\ u\sim v}}\tau_{S,p}(o,u)\cdot p\cdot \tau_{\Lambda,p}(v,x).
\end{equation}
This concludes the proof.
\end{proof} 

We now turn to the proof of Theorem \ref{thm:pre reversed SL}. We follow closely the strategy described in Section \ref{sec:strategy of proof}.

\begin{proof}[Proof of Theorem \textup{\ref{thm:pre reversed SL}}] We fix $S\subset \Lambda\subset \mathbb Z^d$ and $o\in S$ and $x\in \Lambda$. We first prove the result when $\Lambda$ is finite. The general result follows by approximation (using the monotone convergence theorem). 

Recall the definition of the event $\mathcal E_{uv}$ from Section \ref{sec:strategy of proof}: if $u\in S$ and $v\notin S$ with $u\sim v$, 
\begin{equation}
	\mathcal E_{uv}=\{o\connect{S\:}u\}\cap \{\omega_{uv}=1\}\cap \{v\connect{\Lambda\:}x\}\cap \{o\connect{\omega^{[uv]}\cap \Lambda\:}x\}^c,
\end{equation}
where $\omega^{[uv]}$ is the percolation configuration obtained from $\omega$ by closing the edge $\{u,v\}$. We make three simple observations:
\begin{enumerate}
	\item[$(i)$] the following inclusion holds: $\{o\connect{\Lambda\:}x\}\setminus\{o\connect{S\:}x\}\supset \bigcup_{\substack{u\in S\\v\notin S\\u\sim v}}\mathcal E_{uv}$;
	\item[$(ii)$] one has $\mathcal E_{uv}=\{o\connect{S\:}u\}\cap \{\omega_{uv}=1\} \cap\{v\connect{}x \textup{ in }\Lambda\setminus \mathcal C^{[uv]}(u;\Lambda)\}$ where $\mathcal C^{[uv]}(u;\Lambda)$ denotes the cluster of $u$ in $\omega^{[uv]}\cap \Lambda$;
	\item[$(iii)$] the events $(\mathcal E_{uv})_{u\in S, v\notin S, u\sim v}$ are incompatible.
\end{enumerate}
Let us briefly justify these facts. The first point $(i)$ is immediate, and so is $(ii)$ once we have noticed that if $v$ or $x$ belongs to $A$ then $\{v\connect{}x\text{ in }\Lambda\setminus A\}$ cannot occur. We turn to the proof of $(iii)$. 
If $\mathcal E_{uv}\cap \mathcal{E}_{u'v'}$ occurs, any open path from $o$ to $x$ has to go through $\{u,v\}$ \textit{and} $\{u',v'\}$, and the removal of one of the two edges should destroy the connection between $0$ and $x$. This situation cannot occur as illustrated in Figure \ref{fig:event E}. 
\begin{figure}
	\begin{center}
		\includegraphics{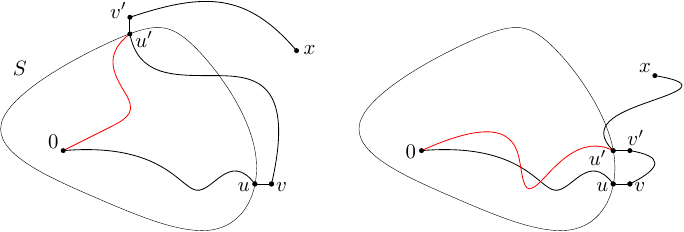}
		\caption{An illustration of why $\mathcal E_{uv}$ and $\mathcal E_{u'v'}$ cannot occur simultaneously. The bold lines represent open paths in a percolation configuration that would satisfy both events. We observe that in each case, we can close $\{u,v\}$ without destroying the connection between $0$ and $x$.}
		\label{fig:event E}
	\end{center}
\end{figure}

Combining $(i)$ and $(iii)$ yields
\begin{equation}\label{eq:pre reversed SL1}
	\tau_{\Lambda,p}(o,x)-\tau_{S,p}(o,x)=\mathbb P_p[\{o\connect{\Lambda\:}x\}\setminus\{o\connect{S\:}x\}]\geq \sum_{\substack{u\in S\\v\notin S\\u\sim v}}\mathbb P_p[\mathcal E_{uv}].
\end{equation}
It remains to express $\mathbb P_p[\mathcal E_{uv}]$ for $u,v$ as above. For this, we use $(ii)$ above. Decomposing the probability according to the value of $\mathcal C^{[uv]}(u;\Lambda)$, and observing that the three events then obtained are independent (because they are measurable with respect to disjoint sets of edges) gives that
\begin{align}\notag
	\mathbb P_p[\mathcal E_{uv}]&=\sum_{A}\mathds{1}\{o\connect{A\cap S\:}u\}\mathbb P_p[\mathcal C^{[uv]}(u;\Lambda)=A]\cdot p\cdot \tau_{\Lambda\setminus A,p}(v,x)
	\\&= \mathbb E_p\Big[\mathds{1}\{o\connect{S\:}u\}\cdot p\cdot \tau_{\Lambda \setminus\mathcal C^{[uv]}(u;\Lambda),p}(v,x)\Big]\notag
	\\&= \mathbb E_p\Big[\mathds{1}\{o\connect{S\:}u\}\cdot p\cdot \tau_{\Lambda \setminus\mathcal C(u;\Lambda),p}(v,x)\Big],\label{eq:pre reversed SL2}
\end{align}
where in the first equality we abused notation by viewing $\mathcal C^{[uv]}(u;\Lambda)$ as both the edge and the vertex cluster of $u$. Plugging \eqref{eq:pre reversed SL2} into \eqref{eq:pre reversed SL1} concludes the proof.
\end{proof}

\section{Regularity, escapability, and an effective version of Theorem \ref{thm:pre reversed SL}}\label{sec:reg,escap,effective}
In this section, we present the effective version of Theorem \ref{thm:pre reversed SL} that we will use throughout to analyse the mean-field regime of Bernoulli percolation. In the rest of the paper, we assume that $d>6$ and that \eqref{eq: 2pt full space estimate} holds. We start by introducing the key concepts of \emph{regularity} and \emph{escapability}, and then prove that, on average, pioneers satisfy these properties.
\subsection{Regular and escapable points}\label{sec: subsec reg esc}
For $\Lambda\subset \mathbb Z^d$ and $u\in \Lambda$,
we write $\mathcal C(u;\Lambda)$ for the cluster of $u$ in $\Lambda$. For every $s>0$ and $T>0$, we consider the event 
\begin{equation}
	\mathcal T_{s,T}^\Lambda(u) :=  \Big\{ |\mathcal C(u;\Lambda)\cap \Lambda_s(u)| \leq T s^4 (\log s)^7\Big\}. 
\end{equation}
When $\Lambda=\mathbb Z^d$ we drop it from the notation. 
\begin{Def}[Regular points] Given $\Lambda\subset \mathbb Z^d$ and $K,T>0$, we call $u\in \Lambda$ a \emph{$(K,T)$-regular point in $\Lambda$} if the events $\mathcal T_{s,T}^\Lambda(u)$ hold for every $s\ge K$. When $\Lambda=\mathbb Z^d$, we refer to $(K,T)$-regular points for short.
\end{Def}
\begin{Rem}\label{rem: regular points} $(1)$ Throughout the paper, the parameter $T$ takes two fixed values $T=1$ or $T=20$. 

$(2)$ Note that $(K,T)$-regular points are $(K,T)$-regular in $\Lambda$ for every 
$\Lambda \subset \mathbb{Z}^d$. 
\end{Rem}

We denote by $\mathcal{P}_{S,\Lambda}^{K\textup{-reg}}$ the number of pioneers of $S$ that are $(K,1)$-regular in $\Lambda$. Again, when $\Lambda=\mathbb Z^d$, we drop it from the notation. 
As seen above, regular points are natural to consider when trying to reverse the Simon--Lieb inequality. We will need an additional technical definition below, see Figure \ref{fig:esc} for an illustration. 

\begin{Def}[Escapable points]\label{def: escap}Let $S\subset \Lambda \subset \mathbb Z^d$ with $S$ containing $0$. Given $K>0$ and $u\in \partial S$, we say that $u$ is \emph{$K$-escapable in $\Lambda$} if there exists $v\in \Lambda\setminus S$ with $v\sim u$ which satisfies the following: there exists $w\in \Lambda_{K}(v)\cap \Lambda$ with $\mathrm{d}(w,\mathcal C(u;\Lambda))> \tfrac{K}{10}$, and a path $\gamma:v\rightarrow w$ inside $\Lambda_K(v)\cap\Lambda$ such that $\gamma$ avoids $\mathcal C(u;\Lambda)$. When $\Lambda=\mathbb Z^d$, we refer to $K$-escapable points for short.
\end{Def}
\begin{figure}[htb]
	\begin{center}
		\includegraphics[scale=1.1]{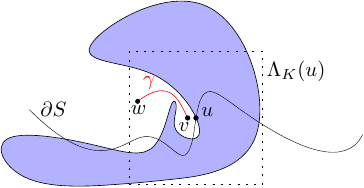}
		\caption{An illustration of a vertex $u$ that is $K$-escapable. The blue region represents $\mathcal C(u)$.}
		\label{fig:esc}
	\end{center}
\end{figure}

We consider the set $\mathcal X^{(K,T)}_{S,\Lambda}$ of pioneers of $S$ that are $(K,T)$-regular and $K$-escapable in $\Lambda$, and let $\mathcal X^{K}_{S,\Lambda}=\mathcal X^{(K,20)}_{S,\Lambda}$. If $u\in\mathcal X_{S,\Lambda}^{K}$, we denote by $(\mathbf{v},\mathbf{w},\boldsymbol{\gamma})$ the earliest triplet (according to some arbitrary order) given by the $K$-escapability in $\Lambda$ of $u$. We say that a realisation $A$ of $\mathcal C(u;\Lambda)$ is $u$-\emph{good} if\footnote{We note that this definition also depends on the pair $(S,\Lambda)$. However, to lighten notations we do not include them. In what follows, it will be clear from context what $S,\Lambda$ are.} it is such that $\mathbb P_p[\mathcal C(u;\Lambda)=A]>0$ and $u\in \mathcal X_{S,\Lambda}^K$. We stress that $(\mathbf{v},\mathbf{w},\boldsymbol{\gamma})$ is entirely measurable in terms of $\mathcal C(u;\Lambda)$. When $\Lambda = \mathbb{Z}^d$, we drop it from the notation.

The notion of escapable points will help to create \emph{open diagrams} in the upcoming analysis, see for instance \eqref{eq:pf main rep5} below. This is often an essential technical step in the study of Bernoulli percolation in dimensions $d>6$ (see for instance the notion of $K$-good line in \cite{Asselah2025capacity}).

With these definitions in hand, we can state the effective version of Theorem \ref{thm:pre reversed SL} we will use in this paper.

\begin{Prop}\label{prop: effective reversed SL} Let $K\geq 1$. There exists $c=c(K,d)>0$ such that the following holds. For every $p\in [p_c/2,p_c]$, and every $S\subset \Lambda \subset \mathbb Z^d$ with $S$ containing $0$ and $x\in \Lambda$,\begin{equation}\label{eq:effective SL 1}
	\tau_{\Lambda,p}(0,x)-\tau_{S,p}(0,x)\geq c\sum_{\substack{u\in \partial S}}\mathbb E_p\Big[\mathds{1}\{u \in \mathcal X^K_{S,\Lambda}\}\cdot\tau_{\Lambda\setminus \mathcal C(u;\Lambda),p}(\mathbf{w},x)\Big].
\end{equation}
In particular,
\begin{equation}\label{eq:effective SL 2}
		\tau_{\Lambda,p}(0,x)-\tau_{S,p}(0,x)\geq c\sum_{\substack{u\in \partial S}}\mathbb E_p\Big[\mathds{1}\{u \in \mathcal X^K_{S,\Lambda}\}\cdot\Big(\tau_{\Lambda,p}(\mathbf{w},x)-\mathsf E_\Lambda(\mathbf{w},x;\mathcal C(u;\Lambda);p)\Big)\Big],
\end{equation}
where for $A\subset \mathbb Z^d$,
\begin{equation}\label{eq:def additive error term}
	\mathsf E_\Lambda(\mathbf{w},x;A;p)=\sum_{z\in A}\tau_{\Lambda,p}(\mathbf{w},z)\tau_{\Lambda,p}(z,x).
\end{equation}
\end{Prop}
\begin{proof} Applying Theorem \ref{thm:pre reversed SL}, we obtain
\begin{align}\notag
	\tau_{\Lambda,p}(0,x)-\tau_{S,p}(0,x)&\geq p\sum_{\substack{u\in \partial S\\ v\notin S\\u\sim v}}\mathbb E_p\Big[\mathds{1}\{0\connect{S\:}u\}\cdot  \tau_{\Lambda\setminus \mathcal C(u;\Lambda),p}(v,x)\Big]\\&\geq \frac{p_c}{2}\sum_{u\in \partial S}\mathbb E_p\Big[\mathds{1}\{u\in \mathcal X_{S,\Lambda}^K\}\cdot \tau_{\Lambda\setminus \mathcal C(u;\Lambda),p}(\mathbf{v},x)\Big]\notag
	\\&=\frac{p_c}{2}\sum_{u\in \partial S}\sum_{A \textup{ $u$-good}}\mathbb P_p[\mathcal C(u;\Lambda)=A]\tau_{\Lambda\setminus A,p}(\mathbf{v},x).\label{eq:pf prop effective rev SL 1}
\end{align}
Under the event $\{\mathcal C(u;\Lambda)=A\}$ for a $u$-good $A$, the $K$-escapability (in $\Lambda$) property of $u$ provides a triplet $(\mathbf{v},\mathbf{w},\boldsymbol{\gamma})$ which satisfies the properties of Definition \ref{def: escap}. In particular,
\begin{equation}\label{eq:pf prop effective rev SL 2}
	\tau_{\Lambda\setminus A,p}(\mathbf{v},x)\geq \mathbb P_p[\boldsymbol{\gamma} \textup{ is open}]\cdot\tau_{\Lambda\setminus A,p}(\mathbf{w},x)\geq c_1 \cdot \tau_{\Lambda\setminus A,p}(\mathbf{w},x),
\end{equation}
where we used the FKG inequality \eqref{eq:FKG} in the first inequality, and where $c_1=c_1(K,d)>0$. This proves \eqref{eq:effective SL 1}. 

We turn to the proof of \eqref{eq:effective SL 2}. For $u \in \partial S$ fixed and a fixed $u$-good $A$ (and its associated $\mathbf{w}$), we observe that 
\begin{equation}
\{\mathbf{w}\connect{\Lambda\:}x\}\setminus\{\mathbf{w}\connect{\Lambda\setminus A\:}x\}\subset \bigcup_{z\in A}\{\mathbf{w}\connect{\Lambda\:}z\}\circ\{z\connect{\Lambda\:}x\},
\end{equation}
so that the BK inequality \eqref{eq:BK ineq} gives 
\begin{equation}\label{eq:pf prop effective rev SL 3}
	\tau_{\Lambda,p}(\mathbf{w},x)-\tau_{\Lambda\setminus A,p}(\mathbf{w},x)\leq \sum_{z\in A}\tau_{\Lambda,p}(\mathbf{w},z)\tau_{\Lambda,p}(z,x)=\mathsf{E}_{\Lambda}(\mathbf{w},x;A;p).
\end{equation}
Plugging \eqref{eq:pf prop effective rev SL 2} and \eqref{eq:pf prop effective rev SL 3} in \eqref{eq:pf prop effective rev SL 1} and setting $c(K,d):=(p_c/2)\cdot c_1$ concludes the proof. 
\end{proof}

\subsection{On average, pioneers are regular and escapable}\label{sec: subsec average}
In Proposition~\ref{prop: effective reversed SL}, pioneers of $S$ were required 
to belong to $\mathcal{X}^K_{S,\Lambda}$. This restriction could a priori be 
very costly. Our next result shows that this is not the case as, on average, pioneers are regular and escapable. Throughout this subsection we take $\Lambda = \mathbb{Z}^d$. Extensions 
to other choices of $(S,\Lambda)$ are discussed in Remarks~\ref{rem: pioneers are regular in average can be made more general} and \ref{rem:comparing mathcal X to pioneers} below. 

 \begin{Prop}\label{prop: in average pioneers are regular and escapable} There exists $K_0>0$ such that the following holds. For every $K\geq K_0$, there exists $c_0=c_0(K,d)>0$ such that, for every $p\in [p_c/2,p_c]$, and every finite set $S\subset \Lambda_{L(p)}$ containing $0$ (with the convention $\Lambda_{L(p_c)} = \mathbb Z^d$),
 \begin{equation}
 	\mathbb E_p\big[|\mathcal X^K_{S}|\big]\geq c_0\cdot \mathbb E_p[\mathcal P_S].
 \end{equation}
 \end{Prop}
 
 We will prove this result using the two following lemmas. 

\begin{Lem}\label{lem: regular are in average also escapable } There exists $K_1>0$ such that the following holds. For every $K\geq K_1$, there exists $c_1=c_1(K,d)>0$ such that, for every $p\in [p_c/2,p_c]$, every $S\subset \mathbb Z^d$ containing $0$,
\begin{equation}
	\mathbb E_p\big[|\mathcal X^K_S|\big]\geq c_1\cdot \mathbb E_p\big[\mathcal{P}_S^{K\textup{-reg}}\big].
\end{equation}
\end{Lem}
\begin{Lem}\label{lem: pioneers are in average regular} There exists $K_2>0$ such that the following holds. For every $K\geq K_2$, every $p\in [p_c/2,p_c]$, every finite $S\subset \Lambda_{L(p)}$ containing $0$,
\begin{equation}
	\mathbb E_p\big[\mathcal{P}_S^{K\textup{-reg}}\big]\geq \frac{1}{2}\cdot \mathbb E_p[\mathcal{P}_S].
\end{equation}
\end{Lem}
\begin{Rem}\label{rem: pioneers are regular in average can be made more general}Lemma \ref{lem: pioneers are in average regular} also gives the following result: for every $K\geq K_2$, every $p\in [p_c/2,p_c]$, every finite $S\subset \Lambda_{L(p)}$ containing $0$, and every $\Lambda\supset S$,
\begin{equation}
	\mathbb E_p\big[\mathcal{P}_{S,\Lambda}^{K\textup{-reg}}\big]\geq \frac{1}{2}\cdot \mathbb E_p[\mathcal{P}_S].
\end{equation}
Indeed, this follows from the fact that $\mathcal{P}_{S,\Lambda}^{K\textup{-reg}}\geq \mathcal{P}_{S}^{K\textup{-reg}}$ (see (2) in Remark \ref{rem: regular points}).
\end{Rem}
\begin{proof}[Proof of Proposition~\textup{\ref{prop: in average pioneers are regular and escapable}}] Fix $p,S$ as in the statement. Let $K_0:=K_1\vee K_2$ where $K_1$ and $K_2$ are respectively given by Lemmas \ref{lem: regular are in average also escapable } and \ref{lem: pioneers are in average regular}. Fix $K\geq K_0$. Let also $c_1=c_1(K,d)$ be given by Lemma \ref{lem: regular are in average also escapable }. Combining these two lemmas,
\begin{equation}
	\mathbb E_p\big[|\mathcal X^K_S|\big]\geq c_1\cdot \mathbb E_p\big[\mathcal{P}_S^{K\textup{-reg}}\big]\geq \frac{c_1}{2}\cdot \mathbb E_p[\mathcal{P}_S].
\end{equation}
This concludes the proof by setting $c_0:=c_1/2$.
 \end{proof}
\subsubsection{Proof of Lemma \ref{lem: regular are in average also escapable }}
We turn to the proof of Lemma \ref{lem: regular are in average also escapable }. The idea is to perform a local surgery that turns a $(K,1)$-regular pioneer into a $K$-escapable and $(K,20)$-regular pioneer.
\begin{proof}[Proof of Lemma \textup{\ref{lem: regular are in average also escapable }}] Let $p\in [p_c/2,p_c]$. Write 
\begin{equation}
	\mathbb E_p\big[|\mathcal X^K_S|\big]=\sum_{u\in \partial S}\mathbb P_p[\{0\connect{S\:}u\}\cap \{u\text{ is $(K,20)$-regular}\}\cap  \{u \text{ is $K$-escapable}\}].
\end{equation}
Let $u\in \partial S$. We claim, that if $K$ is large enough, there exists $C_1=C_1(K,d)>0$ such that
\begin{multline}\label{eq:proof main lemma 2}
	\mathbb P_p[\{0\connect{S\:}u\}\cap \{u\text{ is $(K,1)$-regular}\}]\\\leq C_1\sum_{v\in \Lambda_K(u)\cap \partial S}\mathbb P_p[\{0\connect{S\:}v\}\cap \{v\text{ is $(K,20)$-regular}\}\cap  \{v \text{ is $K$-escapable}\}],
\end{multline}
from which Lemma \ref{lem: regular are in average also escapable } follows readily. 

We turn to the proof of \eqref{eq:proof main lemma 2}. Let $\mathcal A:=\{0\connect{S\:}u\}\cap\{u \textup{ is $(K,1)$-regular}\}$ and $\mathcal B:=\bigcup_{v\in \Lambda_K(u)\cap \partial S}\{0\connect{S\:}v\}\cap\{v\textup{ is $(K,20)$-regular}\}\cap \{v\textup{ is $K$-escapable}\}$. We construct a map $\varphi:\mathcal A\rightarrow \mathcal B$ such that $\omega$ and $\varphi(\omega)$ only differ in $\Lambda_K(u)$, see Figure \ref{fig:localmod} for an illustration. For simplicity, we only describe the construction for the nearest-neighbour case, but it extends straightforwardly to the spread-out setting.

\begin{figure}[htb]
	\begin{center}
	\includegraphics{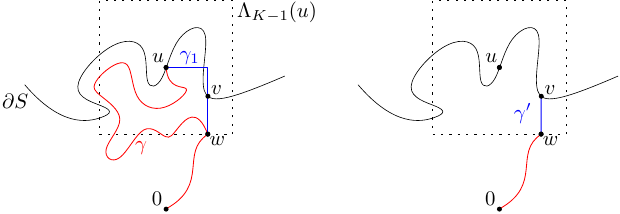}
	\caption{An illustration of the map $\varphi$. On the left: a depiction of $\gamma,\gamma_1,w,v$, which can be constructed from a configuration $\omega\in \mathcal A$. On the right: a depiction of $\varphi(\omega)$, which is obtained from $\omega$ by (essentially) opening the edges of $\gamma'$ and closing all the other edges in $\Lambda_{K-1}(u)$. It is clear from construction that $v$ is always $K$-escapable in $\varphi(\omega)$.}
	\label{fig:localmod}
	\end{center}
\end{figure}

Let $\omega\in \mathcal A$. Let $\gamma$ be any self-avoiding path from $u$ to $0$, that lies in $S$, and that is open in $\omega$. Denote by $w$ the last vertex of $\partial \Lambda_{K-1}(u)$ visited by $\gamma$ before reaching $0$. Denote by $\gamma_{[w,0]}$ the portion of $\gamma$ between $w$ and $0$. Since $w\in \partial \Lambda_{K-1}(u)$, we can assume (as in Figure \ref{fig:localmod}), without loss of generality, that $w_2=u_2-(K-1)$. 
Consider a geodesic path $\gamma_1:u\rightarrow w$ in $\Lambda_{K-1}(u)$. Let $v$ be the last point along $\gamma_1$ that belongs to $\partial S$, and denote by $\gamma'$ the portion of $\gamma_1$ between $v$ and $w$. We make three important observations: $(i)$ $\gamma'$ is included in $S$; $(ii)$ since $w_2\leq u_2-(K-1)$, $\gamma'$ is entirely contained in $\{z: z_2\leq u_2\}$; $(iii$) $|\gamma'|$, the number of edges visited by $\gamma'$, satisfies $|\gamma'|\leq |\gamma_1|=\Vert w-u\Vert_1\leq d\cdot (K-1)$.

The configuration $\varphi(\omega)$ is constructed as follows:
\begin{enumerate}
	\item[$\bullet$] $\varphi(\omega)$ coincides with $\omega$ outside of $\Lambda_K(u)$.
	\item[$\bullet$] In $\Lambda_K(u)$, we open the edges of $\gamma'$ and $\gamma_{[w,0]}\cap \Lambda_K(u)$ and close all the other edges.
	 \end{enumerate}
We claim that $\varphi(\omega)\in \{0\connect{S\:}v\}\cap\{v\textup{ is $(K,20)$-regular}\}\cap \{v\textup{ is $K$-escapable}\} \subset \mathcal B$. Indeed, observe first that $\{0\connect{S\:}v\}$ occurs since the concatenation of $\gamma'$ and $\gamma_{[w,0]}$ is an open path from $v$ to $0$ in $S$. Then, observe that $v$ is $(K,20)$-regular for $K$ large enough. Indeed, denoting by $\mathcal C_{\varphi(\omega)}(v)$ (resp. $\mathcal C_\omega(u)$) the cluster of $v$ (resp. $u$) in $\varphi(\omega)$ (resp. $\omega$), we have by construction that 
 for every $k\geq K$,
\begin{equation}\label{eq:comparing cluster omega to cluster phi(omega)}
	|\mathcal C_{\varphi(\omega)}(v)\cap \Lambda_k(v)|\leq |\mathcal C_{\omega}(u)\cap \Lambda_{2k}(u)|+|\gamma'|\leq (2k)^4(\log 2k)^7+dK\leq 20k^4(\log k)^7,
\end{equation}
where the last inequality holds if $K$ is large enough. Finally, it is easy to see, from the construction, that $v$ is $K$-escapable in $\varphi(\omega)$. Now that $\varphi$ is defined, we write
\begin{align}\notag
	\mathbb P_p[\mathcal A]=\sum_{\omega\in \mathcal A}\mathbb P_p[\omega]\leq C_2(K,d)\sum_{\omega\in \mathcal A}\mathbb P_p[\varphi(\omega)]
	&= C_2(K,d)\sum_{\omega'\in \mathcal B}|\{\omega: \varphi(\omega)=\omega'\}|\mathbb P_p[\omega']
	\\&\leq C_3(K,d)\sum_{\omega'\in \mathcal B} \mathbb P_p[\omega']
	=C_3(K,d)\mathbb P_p[\mathcal B],\label{eq:proof main lemma 3}
\end{align}
where we used that $\omega$ and $\varphi(\omega)$ only differ in $\Lambda_K(u)$ and the fact that $p\in [p_c/2,p_c]$ in the first and last inequalities. The proof of \eqref{eq:proof main lemma 2} follows from \eqref{eq:proof main lemma 3} after a union bound.
\end{proof}
\begin{Rem}\label{rem:comparing mathcal X to pioneers} The construction above is robust and allows for additional requirements to 
be imposed in the definition of $K$-escapability. For instance, one may require 
certain coordinate restrictions on $\mathbf{w}$, which allows to consider a restricted notion of $K$-escapability and extend Lemma \ref{lem: regular are in average also escapable } to some $\mathcal{X}^K_{S,\Lambda}$ with $\Lambda \neq 
\mathbb{Z}^d$. This will be the case in Section~\ref{sec:sharp length} where we analyse $\mathcal{X}^K_{\mathbb{H}_{j}, \mathbb{H}_{j+1}}$ for $j \geq 0$, and require the triplet $(\mathbf{v}, \mathbf{w}, 
\boldsymbol{\gamma})$ given by the $K$-escapability in $\mathbb H_{j+1}$ to satisfy $\mathbf{v}_1 = \mathbf{w}_1 = -(j+1)$ and 
$\boldsymbol{\gamma} \subset \mathbb{H}_{j+1}$. Additionally, we will use this in Section \ref{sec:nc lb} to impose that if $u\in \mathcal X^K_{\Lambda_k}$ is such that $u_i=\pm k$ (in the nearest-neighbour case), then its associated $\mathbf{w}$ satisfies $\mathbf{w}=u\pm\lfloor\tfrac{K}{2}\rfloor \mathbf{e}_i$. \end{Rem}

\subsubsection{Proof of Lemma \ref{lem: pioneers are in average regular}}
To prove Lemma \ref{lem: pioneers are in average regular}, we adapt the argument of \cite[Section~4.4]{panis2025sharp}. We first introduce a \emph{local} regularity condition. To be more precise, for every $s>0$, and every $u\in \mathbb Z^d$, introduce the event \begin{align}\label{eq:def ts loc}
\begin{split}
\mathcal T_{s}^{\textrm{loc}}(u) :=& \Big\{|\mathcal C\big(y; \Lambda_{s^{2d}}(u)\big)\cap \Lambda_s(u)|\leq s^4 (\log s)^4,\ \forall y \in \Lambda_s(u)\Big\} \\ 
				&\cap \Big\{ \exists  \text{ at most }  (\log s)^3 \text{ disjoint open paths from }\Lambda_s(u) \text{ to } \partial \Lambda_{s^{2d}}(u)\text{ in }\Lambda_{s^{2d}}(u)\Big\}.	
				\end{split}		
\end{align}
The interest of this event, compared to $\mathcal T_{s,1}(u)$, is that it only depends on the configuration of the percolation inside the box $\Lambda_{s^{2d}}(u)$, and is thus a \emph{local} event, whereas to determine whether $\mathcal T_{s,1}(z)$ holds or not, one needs to reveal the configuration in the whole $\mathbb Z^d$.  The following simple fact holds by construction (see Claim 4.1 in~\cite{KozmaNachmias} or Claim 5.4 in~\cite{Asselah2025capacity}). 
\begin{Lem}\label{lem:Tsloc.Ts}
One has, for every $s>0$ and every $u\in \mathbb Z^d$, \begin{equation}
\mathcal T_{s}^{\textup{loc}}(u) \subseteq \mathcal T_{s,1}(u).
\end{equation} 
\end{Lem} 
The event $\mathcal T_s^{\textup{loc}}(u)$ is a priori less likely than $\mathcal T_{s,1}(u)$. However, the following result shows that it is still extremely likely. 
\begin{Lem}\label{lem:Tsloc} Let $d>6$ and assume that \eqref{eq: 2pt full space estimate} holds.
There exists $c>0$ such that, for every $p\leq p_c$, every $s>0$, and every $u\in \mathbb Z^d$, \begin{equation}
\mathbb P_p[\mathcal T_s^{\textup{loc}}(u)] \ge 1 - \exp\big(-c(\log s)^4\big).
\end{equation}
\end{Lem} 
\begin{proof}
For completeness we reproduce the short proof of~\cite[Lemma 4.3]{KozmaNachmias}. First of all, since $\mathcal T^{\textup{loc}}_s(u)$ is a decreasing event, it suffices to prove the result for $p=p_c$, which we assume now. The moment estimates of~\cite[Equation~(4.12)]{AizenmanNumberIncipient1997}, gives the existence of 
a constant $c_1=c_1(d)>0$ such that, for every $s\geq 1$,
\begin{equation}\label{lem KN1}
\mathbb P_{p_c}\big[\exists y\in \Lambda_s(u): |\mathcal C\big(y; \Lambda_{s^{2d}}(u)\big)\cap \Lambda_s(u)|> s^4 (\log s)^4]\le 
\exp\big(-c_1(\log s)^4\big). 
\end{equation} 
On the other hand, using the volume estimate of \cite{BarskyAizenmanCriticalExponentPercoUnderTriangle1991} (see also \cite[Theorem~A.1]{vEGPSperco} for a shorter proof), for any $z\in \partial \Lambda_s(u)$, 
\begin{equation}
\mathbb P_{p_c}\big[z\connect{} \partial \Lambda_{s^{2d}}(u)\big]\le   \mathbb P\big[|\mathcal C(z)|\ge s^{2d}-s\big]
\lesssim \frac{1}{s^d}.
\end{equation}
A union bound therefore shows that for some constant $C_1=C_1(d)>0$, 
\begin{equation}
\mathbb P_{p_c}[\partial\Lambda_s(u) \connect{}  \partial \Lambda_{s^{2d}}(u)] \le 
\sum_{z\in \partial \Lambda_s(u)} \mathbb P_{p_c}[z\connect{}  \partial \Lambda_{s^{2d}}(u)]
 \le \frac {C_1}{s}.  
\end{equation}
Then, the BK inequality \eqref{eq:BK ineq} gives 
\begin{align}\label{lem KN2}
 \nonumber \mathbb P_{p_c}\big[\exists \text{ at least } (\log s)^3 \text{ disjoint open } &\textup{paths from }\partial\Lambda_s(u) \text{ to }  \partial \Lambda_{s^{2d}}(u)\big]\\&\le  (C_1/s)^{(\log s)^3}  
  \le \exp(-c_2(\log s)^4),  
\end{align}
where $c_2=c_2(d)>0$. The result follows by combining~\eqref{lem KN1} and~\eqref{lem KN2}. 
\end{proof}
The last input we need is an adaptation of~\cite[Lemma~1.1]{KozmaNachmias}. Note that it holds in every $d\geq 2$ and does not require \eqref{eq: 2pt full space estimate} to hold.

\begin{Lem}\label{lem:KN} There exist $c,C>0$ such that, for every $p\leq p_c$, every $s\leq 2L(p)$ and every $a,b\in \Lambda_s$,
\begin{equation} 
\tau_{\Lambda_s,p}(a,b)\ge c\exp\big(-C(\log s)^2\big).
\end{equation}
\end{Lem}
\begin{proof} The result at $p=p_c$ is exactly \cite[Lemma~1.1]{KozmaNachmias}. It is easy to extend it to $p<p_c$ and $s<L(p)$. Indeed, the only two inputs in their proof are: $(1)$ invariance under symmetries of the lattice, $(2)$ a uniform lower bound on $\varphi_p(\Lambda_\ell)$ for $\ell\leq s$. The former input is true regardless of the value of $p$, and for the latter one, we have by definition that $\varphi_p(\Lambda_\ell)\geq \tfrac{1}{2}$ for every $\ell\leq s$ since $s<L(p)$. The extension of the result to every $s\leq 2L(p)$ follows from the FKG inequality.
\end{proof} 


We are now equipped to prove Lemma \ref{lem: pioneers are in average regular}. 

\begin{proof}[Proof of Lemma~\textup{\ref{lem: pioneers are in average regular}}] 
 Let $K\ge 1$ to be chosen large enough. Fix $p\in[p_c/2,p_c]$ and $S\subset \Lambda_{L(p)}$ finite and containing $0$. We say that a point $u\in \partial S$ is $s$-locally bad if the event $\mathcal T_s^{\textup{loc}}(u)$ does not hold. We denote by $\mathcal{P}_S^{s\textup{-loc-bad}}$ the number of points on $\partial S$ which are 
$s$-locally bad and connected to $0$ in $S$. 
Due to Lemma~\ref{lem:Tsloc.Ts}, one has  
\begin{equation}\label{eq:Xns}
\mathbb E_p\big[\mathcal{P}_S^{K\textrm{-reg}}\big] \ge \mathbb E_p[\mathcal{P}_S] - \sum_{s\ge K} \mathbb E_p\big[\mathcal{P}_S^{s\textup{-loc-bad}}\big].
\end{equation} 

We now bound the sum of \eqref{eq:Xns}. Fix $s\ge K$ and consider the set $\mathcal U = \{ u \in \mathbb Z^d : u_i \in \{0,s^{2d}\}, \ \forall i =1,\dots,d\}$. For $w\in \mathcal U$, we introduce
\begin{equation}
\mathcal B(w) = \{\Lambda_{2s^{2d}}(z) : z\in w+(4s^{2d}+1)\cdot \mathbb Z^d\}.
\end{equation} 
We denote by $\mathcal{Q}(w)$ the set of boxes in this partition that intersect 
$\partial S$. We explore them using the following algorithm. We first reveal the 
percolation configuration outside the union of these boxes. Then, as long as there 
exists an unexplored box in $\mathcal{Q}(w)$ that either contains $0$ or is 
connected to $0$ by an open path lying entirely in $S$ in the already-explored region, 
we select one such box and reveal the configuration inside it. We let 
$\mathcal{N}(w)$ denote the number of boxes in $\mathcal{Q}(w)$ revealed by this 
algorithm.

Observe that, throughout the exploration, whenever the configuration inside a new box is revealed, the following holds almost surely. Conditional on the configuration outside this box, the probability that the box contains a pioneer is at least $c_2 \exp(-C_2(\log s)^2)$, for some constants $c_2,C_2>0$ depending only on $d$. To see this, let $z$ be such that $B=\Lambda_{2s^{2d}}(z)$ intersects $\partial S$. Suppose first that $B$ does not contain $0$---the case $0\in B$ being analogous---and that the algorithm has selected $B$. This means that the exploration has revealed a collection of boundary points $x_1,\ldots,x_\ell\in \partial B$ which are connected to $0$ in $S\setminus B$.
Let $b\in B\cap \partial S$. Then, 
\begin{multline}
	\mathbb P_p[\textup{There exists a pioneer in $B$}\:|\:\textup{Exploration before $B$}]\\\geq \tau_{B,p}(x_1,b)\geq c_1\exp\big(-C_1(\log s)^2\big),
\end{multline}
where in the first inequality we used the fact that every open path from $x_1$ to $b$ in $B$ must intersect $\partial S$ at least once so that the first such point is  a pioneer, and where the second inequality follows from Lemma \ref{lem:KN}, using that since $S\subset \Lambda_{L(p)}$, there exists a translate of $\Lambda_\ell$ with $\ell\leq 2L(p)$ that is contained in $B$ and contains $x_1$ and $b$.
  As a consequence, for every $w\in \mathcal U$, one has
\begin{equation}
\mathbb E_p[\mathcal{P}_S] \ge c_1\exp\big(-C_1(\log s)^2\big) \cdot \mathbb E_p[\mathcal N(w)], 
\end{equation} 
and hence also, since $|\mathcal U|=2^d$, 
\begin{equation}\label{eq:Xneps}
\mathbb E_p[\mathcal{P}_S] \ge \frac{c_1}{2^d} \exp\big(-C_1(\log s)^2\big) \cdot \sum_{w\in \mathcal U} \mathbb E_p[\mathcal N(w)]. 
\end{equation}

For a box $q\in \mathcal Q(w)$, call \emph{interior} of $q$ the set of points in $q$ which are at distance at least $s^{2d}$ from the boundary of $q$. Observe that as $w$ varies in $\mathcal U$, the union of all the interiors of the boxes $q\in \mathcal Q(w)$ covers $\mathbb Z^d$, and in particular the whole boundary $\partial S$.   
Given a point $u$ on $\partial S$ which is in the interior of a box $q\in \mathcal Q(w)$, the event $\mathcal T_s^{\textup{loc}}(u)$ only depends on the configuration of edges inside $q$. Since there are $\lesssim s^{2d^2}$ such points in each box $q\in \mathcal Q(w)$, a union bound and Lemma~\ref{lem:Tsloc} give that, for some constants $c_2,C_2>0$, and for any $s$ as above,
\begin{equation}  
\mathbb E_p\big[\mathcal{P}_S^{s\textup{-loc-bad}}\big] \le C_2 s^{2d^2}\exp\big(-c_2(\log s)^4\big) \sum_{w\in \mathcal U}  \mathbb E_p[\mathcal N(w)].
\end{equation}
Using~\eqref{eq:Xneps} and taking $K$ large enough yields that 
\begin{equation}\label{eq:proof comp 4}
\sum_{s\ge K}\mathbb E_p\big[\mathcal{P}_S^{s\textup{-loc-bad}}\big] \le \frac 12 \cdot \mathbb E_p[\mathcal{P}_S].
\end{equation}
Plugging \eqref{eq:proof comp 4} in \eqref{eq:Xns} concludes the proof.
\end{proof}

\section{Uniform boundedness of $\varphi_{p_c}(S)$}\label{sec:proof bound phi(S)}

We now turn to the proof of Theorem \ref{thm:main}. Given Proposition \ref{prop: in average pioneers are regular and escapable}, it is sufficient to prove the following result. 

\begin{Lem}\label{lem:bounded regular escapable pioneers} There exists $K_3\geq 1$ such that the following holds. For every $K\geq K_3$, there exists $C_0=C_0(K,d)>0$ such that, for every finite $S\subset \mathbb Z^d$ containing $0$, 
\begin{equation}
	\mathbb E_{p_c}\big[|\mathcal X^{K}_S|\big]\leq C_0.
\end{equation}
\end{Lem}

\begin{proof}[Proof of Theorem \textup{\ref{thm:main}}] Let $K_0$ be given by Proposition \ref{prop: in average pioneers are regular and escapable} and $K_3$ be given by Lemma \ref{lem:bounded regular escapable pioneers}. Set $K:=K_0\vee K_3$, and let $c_0(K,d)$ and $C_3=C_3(K,d)$ be given by Proposition \ref{prop: in average pioneers are regular and escapable} and Lemma \ref{lem:bounded regular escapable pioneers} respectively. By these results, for every $S\subset \mathbb Z^d$ finite containing $0$, one has
\begin{equation}
	\varphi_{p_c}(S)\lesssim \mathbb E_{p_c}[\mathcal{P}_S]\leq c_0^{-1}\cdot\mathbb E_{p_c}\big[|\mathcal X^K_{S}|\big]\leq c_0^{-1}C_0=:C,
\end{equation}
where the first inequality follows from \eqref{eq:comparison pioneers and phi}. This gives the result in the case where $S$ is finite. 

The case of infinite $S$ follows by approximation. Let $p<p_c$. On the one hand, 
\begin{equation}
	\limsup_{n\rightarrow \infty}\varphi_p(S\cap \Lambda_n)\leq \limsup_{n\rightarrow \infty}\varphi_{p_c}(S\cap \Lambda_n)\leq C.
\end{equation}
On the other hand, for every $n\geq 1$,
\begin{equation}
	0\leq \varphi_p(S\cap \Lambda_n)-\sum_{\substack{u\in S\\v\notin S\\u \sim v}}\mathds{1}\{u\in \Lambda_n\}\tau_{S\cap \Lambda_n,p}(0,u)\cdot p\leq \varphi_{p}(\Lambda_n).
\end{equation}
Since $p<p_c$, $\varphi_p(\Lambda_n)\rightarrow 0$ as $n$ tends to infinity. Moreover, by the monotone convergence theorem,
\begin{equation}
	\lim_{n\rightarrow \infty}\sum_{\substack{u\in S\\v\notin S\\u \sim v}}\mathds{1}\{u\in \Lambda_n\}\tau_{S\cap \Lambda_n,p}(0,u)\cdot p=\varphi_{p}(S).
\end{equation}
Combining the above displayed equations gives that $\varphi_p(S)\leq C$. Taking the (monotone) limit $p\nearrow p_c$ concludes the proof.
%
%
\end{proof}

Before moving to the proof of Lemma \ref{lem:bounded regular escapable pioneers}, we complete the proofs of Proposition \ref{prop:reversed SL} and Corollary \ref{cor:partial monot}.
\begin{proof}[Proof of Proposition \textup{\ref{prop:reversed SL}}] Let $\varepsilon>0$. Let $S\subset \mathbb Z^d$ finite containing $0$ and $x\in \mathbb Z^d$ 
By Theorem \ref{thm:main}, one has 
\begin{equation}\label{eq:proof prop2.1}
\sum_{\substack{u\in S: \: |u-x|\geq \varepsilon |x|\\v\notin S\\u\sim v}}\tau_{S,p_c}(0,u)\cdot p_c\leq \varphi_{p_c}(S)\lesssim 1.
\end{equation}
If $u\in \partial S$ satisfies $|u-x|\geq \varepsilon |x|$, then by \eqref{eq: 2pt full space estimate}, there exists $c_1=c_1(\varepsilon,d)>0$ such that for every $v\sim u$, one has $c_1 \tau_{p_c}(v,x)\leq \tau_{p_c}(0,x)$.
Combining this observation with \eqref{eq:proof prop2.1} concludes the proof.
\end{proof}

\begin{proof}[Proof of Corollary~\textup{\ref{cor:partial monot}}] Let $p\leq p_c$ and fix $S,\Lambda$ as in the statement. Applying Lemma \ref{lem:SL} gives, for every $x\in \partial\Lambda$,
\begin{equation}
	\tau_{\Lambda,p}(0,x)\leq \tau_{S,p}(0,x)+\sum_{\substack{u\in S\\v\notin S\\ u\sim v}}\tau_{S,p}(0,u)\cdot p \cdot \tau_{\Lambda,p}(v,x)
\end{equation}
Multiplying the above display by $p$ and summing it over $y\notin \Lambda$ with $y\sim x$, and then over $x\in \partial \Lambda$ gives
\begin{equation}
	\varphi_p(\Lambda)\leq \varphi_p(S)+\sup_{z\in \Lambda}\varphi_p(\Lambda-z)\cdot \varphi_p(S) \lesssim \varphi_p(S),
\end{equation}
where the second inequality follows from Theorem \ref{thm:main}. This concludes the proof.
\end{proof}

We now turn to the proof of Lemma \ref{lem:bounded regular escapable pioneers}.

\begin{proof}[Proof of Lemma~\textup{\ref{lem:bounded regular escapable pioneers}}] We let $p=p_c$ and drop it from the notation. We also fix a finite set $S\subset \mathbb Z^d$ containing $0$, and stress that all the constants that appear below are independent of $S$. Let $K\geq 1$ and let $c_1=c_1(K,d)>0$ be given by Proposition \ref{prop: effective reversed SL}. Applying \eqref{eq:effective SL 1} of Proposition \ref{prop: effective reversed SL} to $K$, $\Lambda=\mathbb Z^d$, and $S$, we obtain for every $x\in \mathbb Z^d$,
\begin{equation}\label{eq:pf main rep1}
	\tau(0,x)\geq c_1\sum_{\substack{u\in \partial S}}\mathbb E\Big[\mathds{1}\{u \in \mathcal X^K_{S}\}\cdot\tau_{\mathcal C(u)^c}(\mathbf{w},x)\Big]. 
\end{equation}
Combining \eqref{eq:pf main rep1} and Fatou's lemma gives
\begin{equation}\label{eq:pf main rep2}
	c_1^{-1}\geq \sum_{u\in \partial S}\sum_{A\textup{ $u$-good}}\mathbb P[\mathcal C(u)=A]\cdot \liminf_{|x|\rightarrow \infty} \frac{\tau_{A^c}(\mathbf{w},x)}{\tau(0,x)}.
\end{equation}
To conclude the proof, it is sufficient to prove the existence of $c_2=c_2(d)>0$ such that, if $K$ is large enough, for every $u,A$ as above,
\begin{equation}\label{eq:pf main rep3}
\liminf_{|x|\rightarrow \infty} \frac{\tau_{A^c}(\mathbf{w},x)}{\tau(0,x)}\geq c_2,
\end{equation}
which, by the same computation as in the proof of Proposition \ref{prop: effective reversed SL}, is in turn implied by
\begin{equation}\label{eq:pf main rep3.5}
	\liminf_{|x|\rightarrow \infty}\left(\frac{\tau(\mathbf{w},x)}{\tau(0,x)}-\frac{\mathsf{E}(\mathbf{w},x;A)}{\tau(0,x)}\right)\geq c_2.
\end{equation}
Fix $u\in \partial S$ and a $u$-good set $A$. Since $\mathcal C(u)$ is finite almost surely, we may assume that $A$ is finite. Let $|x|\geq 10(\textup{diam}(A)\vee K)$. By \eqref{eq: 2pt full space estimate}, there exist $c_3,C_1>0$ (which only depend on $d$) such that, for every $z\in A$,
\begin{equation}\label{eq:pf main rep4}
	\tau(\mathbf{w},x)\geq c_3\tau(0,x), \qquad \tau(z,x)\leq C_1\tau(0,x).
\end{equation}
Moreover, there exists $C_2=C_2(d)>0$ such that
\begin{equation}
\label{eq:pf main rep5}
	\sum_{z\in A}\tau(\mathbf{w},z)\lesssim \sum_{\ell\geq \log_2(K/10)}\frac{|A\cap (\Lambda_{2^{\ell+1}}(\mathbf{w})\setminus \Lambda_{2^{\ell}}(\mathbf{w}))|}{2^{\ell(d-2)}}\lesssim \sum_{\ell\geq \log_2(K/10)}\frac{2^{4\ell}\ell^7}{2^{\ell(d-2)}} 
	\leq \frac{C_2}{\sqrt{K}},
	\end{equation}
where we used that $A\cap \Lambda_{K/10}(\mathbf{w})=\emptyset$, \eqref{eq: 2pt full space estimate}, and the fact that, since $A$ is $u$-good, for every $k\geq K/10$,
\begin{equation}
	|A\cap (\Lambda_{2k}(\mathbf{w})\setminus \Lambda_{k}(\mathbf{w}))|\leq |A\cap \Lambda_{10k}(u)|\lesssim k^4(\log k)^7.
\end{equation}
Combining \eqref{eq:pf main rep4} and \eqref{eq:pf main rep5} gives
\begin{equation}\label{eq:pf main rep6}
	\mathsf E(\mathbf{w},x;A)=\sum_{z\in A}\tau(\mathbf{w},z)\tau(z,x)\leq \frac{C_1C_2}{\sqrt{K}}\tau(0,x).
\end{equation}
Plugging \eqref{eq:pf main rep4} and \eqref{eq:pf main rep6} into \eqref{eq:pf main rep3.5} gives
\begin{equation}\label{eq:pf main rep7}
\liminf_{|x|\rightarrow \infty}	\left(\frac{\tau(\mathbf{w},x)}{\tau(0,x)}-\frac{\mathsf{E}(\mathbf{w},x;A)}{\tau(0,x)}\right)
\geq 
c_3-\frac{C_1C_2}{\sqrt{K}}.
\end{equation}
Choosing $K$ large enough so that $C_1C_2/\sqrt{K}\leq c_3/2$ concludes the proof.
\end{proof}

\begin{Rem} The computation of \eqref{eq:pf main rep5} justifies the need for the concept of $K$-escapability. This will be used several times below.
\end{Rem}

\section{Analysis of the sharp length}\label{sec:sharp length}

In this section, we prove Theorem \ref{thm:sharplength} and deduce from it the upper bounds in Theorems \ref{thm:nc bounds phi(S)} and \ref{thm:bounds 2pt}. Recall that the susceptibility $\chi(p)$ is defined for $p<p_c$ by
\begin{equation}
	\chi(p)=\sum_{x\in \mathbb Z^d}\tau_p(0,x).
\end{equation}
We analyse $L(p)$ by relating it to $\chi(p)$, which is well-understood by \eqref{eq:susceptibility}.

\begin{Prop}\label{prop.L(p).chi(p)} Let $d>6$ and assume that \eqref{eq: 2pt full space estimate} holds. There exist $c,C>0$ such that, for every $p\in [p_c/2,p_c)$,
\begin{equation}\label{eq:bounds on chi in terms of L(p)}
	cL(p)^2\leq \chi(p)\leq CL(p)^2.
\end{equation}
\end{Prop}
The upper bound is very standard and already appeared in \cite{hutchcroft2022derivation}. The main novelty lies in the derivation of a matching order lower bound.

\begin{proof}[Proof of Theorem~\textup{\ref{thm:sharplength}}] Let $p\in [p_c/2,p_c)$. Combining Proposition \ref{prop.L(p).chi(p)} and \eqref{eq:susceptibility} gives
\begin{equation}
	L(p)^2\asymp \chi(p)\asymp (p_c-p)^{-1}.
\end{equation}
This concludes the proof.
\end{proof}

We now turn to the proof of Proposition \ref{prop.L(p).chi(p)}. Recall that $\mathbb H=\mathbb N\times \mathbb Z^{d-1}$, and for $\ell\in \mathbb Z$, $\mathbb H_\ell=\mathbb H -\ell \mathbf{e}_1$. 

\begin{proof} Fix $p\in[p_c/2,p_c)$.

\vspace{5pt}
\textbf{Proof of the upper bound.} We first prove the upper bound in \eqref{eq:bounds on chi in terms of L(p)}. Applying Lemma \ref{lem:SL} to $p$, $\Lambda=\mathbb Z^d$ and a finite set $S\subset \mathbb Z^d$, and summing it over $x\in \mathbb Z^d$ yields
\begin{equation}\label{eq:pf L(p)1}
	\chi(p)\leq \sum_{x\in \mathbb Z^d}\tau_{S,p}(0,x)+\varphi_p(S)\chi(p).
\end{equation} 
By definition of $L(p)$, there exists $S\subset \Lambda_{L(p)}$ containing $0$ and satisfying $\varphi_p(S)<\tfrac{1}{2}$. Plugging this $S$ in \eqref{eq:pf L(p)1} gives
\begin{equation}
	\chi(p)\leq 2\sum_{x\in \mathbb Z^d}\tau_{p,S}(0,x)\leq 2\sum_{x\in \Lambda_{L(p)}}\tau_{p_c}(0,x)\lesssim L(p)^2,
\end{equation}
where the last inequality follows from \eqref{eq: 2pt full space estimate}. This gives the upper bound in \eqref{eq:bounds on chi in terms of L(p)}.

\vspace{5pt}
\textbf{Proof of the lower bound.} 
 We now turn to the proof of the lower bound in \eqref{eq:bounds on chi in terms of L(p)}. The proof follows a similar strategy as \cite{DumPan24Perco}, but does not rely on any estimate on $\tau_{\mathbb H,p_c}$. This is made possible by the fact that, unlike \cite{DumPan24Perco}, we work with averaged quantities for which sharp estimates are provided by Theorem~\ref{thm:main}. 

Let $x\in \mathbb Z^d$ with $x_1\in \{-\lfloor \tfrac{L(p)}{2}\rfloor,\ldots,0\}$. Thanks to Theorem \ref{thm:pre reversed SL} applied to $\Lambda=\mathbb H_{j+1}$ and $S=\mathbb H_j$ for $j\in \{|x_1|,\ldots, L(p)-1\}$,
\begin{align}
	\tau_p(0,x)&\geq \sum_{j=|x_1|}^{L(p)-1}\Big(\mathbb P_p[0\connect{\mathbb H_{j+1}\:}x]-\mathbb P_p[0\connect{\mathbb H_j\:}x]\Big) \notag
	\\&\geq p\sum_{j=|x_1|}^{L(p)-1}\sum_{\substack{u\in \mathbb H_j\\v\in \mathbb H_{j+1}\setminus  \mathbb H_j\\u\sim v}}\mathbb E_p\Big[\mathds{1}\{0\connect{\mathbb H_j\:}u\}\mathbb P_p[v\connect{}x\textup{ in }\mathbb H_{j+1}\setminus \mathcal C(u;\mathbb H_{j+1})]\Big],\label{eq:pf chi LB 1}
\end{align}
See Figure \ref{fig:proof L(p)} for an illustration.
\begin{figure}[htp]
	\begin{center}
	\includegraphics{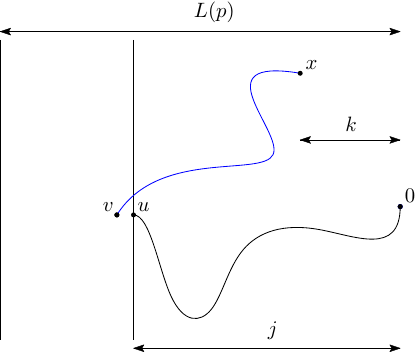}
	\put(-195,8){$\mathbb H_{L(p)-1}$}
		\caption{An illustration of the proof of the bound $\chi(p) \gtrsim L(p)^2$. The idea is to reconstruct $\tau_p(0,x)$ for $x_1=-k$ according to the above picture. The black line (resp.\ blue line) represents the two-point function $\tau_{\mathbb{H}_j, p}(0, u)$ (resp.\ $\tau_{\mathbb{H}_{j+1} \setminus \mathcal C(u;\mathbb H_{j+1}), p}(v, x)$). The parameters satisfy $j \geq k$ and range over an interval of length $\asymp L(p)$. By Theorem~\ref{thm:pre reversed SL}, summing this diagram over $j,k, \{u,v\},x$ yields a lower bound on $\chi(p)$. If the restriction to $\mathcal{C}(u)^c$ could be dropped from the blue line, then by definition of $L(p)$, the sum over $\{u,v\},x$ for fixed $j, k$ with $L(p) > j \geq k$ would be bounded below by a constant $c = c(d) > 0$. The core of the argument, carried out as in Section~\ref{sec:proof bound phi(S)}, is to show that this restriction can indeed be removed, in a suitable averaged sense.}
		\label{fig:proof L(p)}
	\end{center}
\end{figure}

Similarly to what was done in Section~\ref{sec:proof bound phi(S)}, the idea is to restrict the last sum of \eqref{eq:pf chi LB 1} to points $u$ lying in $\mathcal{X}_j^K := \mathcal{X}_{\mathbb{H}_j,\mathbb H_{j+1}}^K$. This approach would yield the desired result, except that the argument breaks down in $d = 7$. The reason for this limitation is the absence of pointwise estimates on the two-point function restricted to a half-space: to conclude, it would suffice to have any improvement over the exponent $d-2$ in \eqref{eq: 2pt full space estimate} for $\tau_{\mathbb{H}, p_c}(0, x)$. Nevertheless, Theorem~\ref{thm:main} tells us that the quantities $\varphi_{p_c}(\mathbb{H}_\ell)$, for $\ell \geq 0$, are uniformly bounded. This means that regions where the half-space two-point function is insufficiently small cannot be too large. We exploit this by imposing additional restrictions on $\mathcal{X}_j^K$. Also, as explained in Remark \ref{rem:comparing mathcal X to pioneers}, we can modify the definition of $K$-escapability in $\mathbb H_{j+1}$ so that, for $u\in \mathcal X_j^K$, the triplet $(\mathbf{v},\mathbf{w},\boldsymbol{\gamma})$ associated with $u$ satisfies $\mathbf{v}_1=\mathbf{w}_1=-(j+1)$.

For $y\in \mathbb H_{j+1}$ with $y_1=-(j+1)$ and $s>0$, define the set
\begin{equation}
\mathcal V_s^j(y) := \{z \in \Lambda_s(y) : \tau_{\mathbb H_{j+1},p}(z,y)\ge s^{\frac 32 -d}\}. 
\end{equation}
This set is expected to be relatively small since the bound $\tau_{\mathbb H_{j+1},p_c}(z,y)\lesssim (1+|z-y|)^{1-d}$ should hold \cite{panis2025sharp}: we expect that there is $\epsilon>0$ such that $|\mathcal V_s^j(y)|\lesssim |\Lambda_s(y)|^{1-\epsilon}$. This motivates the following definition:
\begin{equation}
\widetilde {\mathcal X}_j^K := \mathcal X_j^K \cap \Big\{u \in \partial \mathbb H_j: |\mathcal C(u;\mathbb H_{j+1}) \cap \mathcal V_s^j(y)| \le 20s^{4-\frac{1}{d}}, \quad \forall s\ge K, \ \forall y\in \Lambda_K(u), \: y_1=-(j+1)\Big\}. 
\end{equation}
The constant $20$ here plays the same role as in Section \ref{sec:proof bound phi(S)} with the notion of $(K,20)$-regular points. Compared to Section \ref{sec:proof bound phi(S)}, we now say that a set $A$ is \emph{$u$-super good} if $\mathbb P_p[\mathcal C(u;\mathbb H_{j+1})=A]>0$ and if it is such that $u\in \widetilde{\mathcal X}_j^K$. In particular, if $A$ is $u$-super good, one has  
\begin{equation}\label{eq:additional restriction}
|A\cap \mathcal V_s^j(y)| \le 20s^{4-\frac 1d}, \quad \text{for every } s\ge K, 
\text{ and every }y\in \Lambda_K(u) \textup{ with }y_1=-(j+1). 
\end{equation}
By \eqref{eq:pf chi LB 1} and a straightforward adaptation of Proposition \ref{prop: effective reversed SL}, for every $K$, there exists $c_1=c_1(K,d)>0$ such that,
\begin{equation}\label{eq:pf chi LB 2}
	\tau_p(0,x)\geq c_1\sum_{j=|x_1|}^{L(p)-1}\sum_{u\in \partial \mathbb H_j}\mathbb E_p\Big[\mathds{1}\{u\in \widetilde{\mathcal X}_j^K\}\Big(\tau_{\mathbb H_{j+1},p}(\mathbf{w},x)-\mathsf{E}_{j}(\mathbf{w},x;\mathcal C(u;\mathbb H_{j+1}))\Big)\Big],
\end{equation}
where $\mathsf{E}_{j}(\mathbf{w},x;\mathcal C(u;\mathbb H_{j+1})):=\mathsf{E}_{\mathbb H_{j+1}}(\mathbf{w},x;\mathcal C(u;\mathbb H_{j+1}),p)$ was defined in \eqref{eq:def additive error term}.

To advance the proof and handle the term $\mathsf{E}_{j}(\mathbf{w},x;\mathcal C(u;\mathbb H_{j+1}))$, we rely on the following two claims.
\begin{Claim}\label{claim2} For every $k\geq 0$ and $\ell\in \mathbb Z$ satisfying $k\geq \ell$,
\begin{equation}
	\sum_{z: z_1=-\ell}\tau_{\mathbb H_k,p_c}(0,z)\lesssim 1+  (k-\ell) \wedge k.
\end{equation}
\end{Claim}
\begin{proof}[Proof of Claim \textup{\ref{claim2}}] This claim is exactly the same as \cite[Lemma~2.4]{DumPan24Perco}. For completeness, we reproduce the short argument. We first treat the case $\ell\geq 0$. Decomposing an open self-avoiding path from $0$ to $z$ according to the left-most vertex $y$ it visits, and using the BK inequality \eqref{eq:BK ineq}, gives, for every $z$ such that $z_1=-\ell$,
\begin{equation}\label{eq:pf claim1}
	\tau_{\mathbb H_k,p_c}(0,z)\leq \sum_{m=0}^{k-\ell}\sum_{y:\: y_1=-\ell-m}\tau_{\mathbb H_{\ell+m},p_c}(0,y)\tau_{\mathbb H_{\ell+m},p_c}(y,z).
\end{equation}
Summing \eqref{eq:pf claim1} over $z$ with $z_1=-\ell$, and using translation invariance, gives
\begin{align}
	\sum_{z: z_1=-\ell}\tau_{\mathbb H_k,p_c}(0,z)&\leq (1+k-\ell)\Big(\sup_{n\geq 0}\sum_{y: \: y_1=-n}\tau_{\mathbb H_n,p_c}(0,y)\Big)^2 \\&\lesssim (1+k-\ell) \Big(\sup_{n\geq 0}\varphi_{p_c}(\mathbb H_n)\Big)^2 \notag\lesssim 1+(k-\ell),
\end{align}
where the last inequality follows from Theorem \ref{thm:main}. In the case $\ell<0$, we replace \eqref{eq:pf claim1} by
\begin{equation}
	\tau_{\mathbb H_k,p_c}(0,z)\leq \sum_{m=0}^{k}\sum_{y:\: y_1=-\ell-m}\tau_{\mathbb H_{\ell+m},p_c}(0,y)\tau_{\mathbb H_{\ell+m},p_c}(y,z),
\end{equation}
and follow the same argument as above to get
\begin{equation}
	\sum_{z: z_1=-\ell}\tau_{\mathbb H_k,p_c}(0,z)\lesssim 1+k.
\end{equation}
This concludes the proof of the first claim.
\end{proof}

\begin{Claim}\label{claim1} Let $\varepsilon>0$. There exists $K_4=K_4(\varepsilon,d)>0$ such that the following holds. For every $K\geq K_4$, every $k\in \{0,\ldots,\lfloor \tfrac{L(p)}{2}\rfloor\}$,  every $j\in \{k,\ldots,L(p)-1\}$, every $u\in \partial \mathbb H_j$, and every $u$-super good set $A$ (and its associated $\mathbf{w}$), one has
\begin{equation}
\sum_{x:\:x_1=-k}\mathsf{E}_{j}(\mathbf{w},x;A)\leq \varepsilon.
	\end{equation}
\end{Claim}

\begin{Rem} As discussed above, the result of Claim~\ref{claim1} holds for $u$-good sets $A$ in dimension $d>7$. It is only for the dimension $d=7$ that one needs to introduce the set $\widetilde {\mathcal X}_j^K$ and the notion of $u$-super good sets.  
\end{Rem}

\begin{proof}[Proof of Claim \textup{\ref{claim1}}] Recall that we have modified the definition of $K$-escapability in $\mathbb H_{j+1}$ so that $\mathbf{w}$ satisfies $\mathbf{w}_1=-(j+1)$. Moreover, by definition $\mathrm{d}(\mathbf{w},A)> K/10$. By translation invariance, if $z$ is such that $z_1\geq -(j+1)$,
\begin{align}
	\sum_{x:\: x_1=-k}\tau_{\mathbb H_{j+1},p}(z,x)=\sum_{\tilde{x}:\:\tilde{x}_1=-(k+z_1)\:}\tau_{\mathbb H_{j+1+z_1},p}(0,\tilde{x})&\lesssim 1+(j+1+z_1-(k+z_1))\wedge (j+1+z_1)\notag
	\\&\lesssim 1+(j+1+z_1),
\end{align}
where we used Claim \ref{claim2} in the first inequality.
As a result, we get  
\begin{align}\notag
	\sum_{x:\:x_1=-k}\mathsf{E}_{j}(\mathbf{w},x;A)&=\sum_{x:\:x_1=-k}\sum_{z\in A}\tau_{\mathbb H_{j+1},p}(\mathbf{w},z)\tau_{\mathbb H_{j+1},p}(z,x)\\&\lesssim \sum_{z\in A} \tau_{\mathbb H_{j+1},p}(\mathbf{w},z) \cdot [1+(j+1+z_1)]. 
\end{align}
Now, the fact that $A$ is $u$-super good yields similarly as for~\eqref{eq:pf main rep7}, 
\begin{align}
\notag  \sum_{z\in A} \tau_{\mathbb H_{j+1},p}(\mathbf{w},z)\cdot[1+(j+1+z_1)] & \lesssim \sum_{\ell \ge \log_2(K/10)} 2^\ell \cdot\Big(\frac{|A\cap \Lambda_{2^\ell}(\mathbf{w}) |}{2^{\ell(d-3/2)}} + \frac{|A\cap \mathcal V_{2^\ell}^j(\mathbf{w})|}{2^{\ell (d-2)}}\Big) \\
 &  \lesssim 
 \sum_{\ell \ge \log_2(K/10)} \Big(\frac{2^{5\ell} \ell^7}{2^{\ell(d-3/2)}} + \frac{2^{\ell(5 - 1/d)}}{2^{\ell(d-2)}} \Big) \lesssim \frac 1{K^{1/d}}. 
\end{align}
Since the implicit constants do not depend on $K$, it suffices to choose $K$ large enough to conclude the proof.
\end{proof}

Summing \eqref{eq:pf chi LB 2} over every $x$ with $x_1=-k$ (where we recall that $0\leq k \leq \lfloor \tfrac{L(p)}{2}\rfloor$), using Claim~\ref{claim1} for $\varepsilon>0$ to be fixed, and picking $K\geq K_4(\varepsilon,d)$ gives
\begin{equation}\label{eq:pf chi LB 3}
		c_1^{-1}\sum_{x:x_1 =-k}\tau_p(0,x)\geq \sum_{j=k}^{L(p)-1}\mathbb E_p\big[|\widetilde {\mathcal X}_j^K|\big]\cdot \Big(\sum_{y:\: y_1=-(j+1-k)}\tau_{\mathbb H_{j+1-k}}(0,y)-\varepsilon\Big).
\end{equation}
Observe that there exists $c_2=c_2(d)>0$ such that, for every $\ell\geq 0$, one has (recall that $p\in [p_c/2,p_c)$)
\begin{equation}
	\sum_{z:\:z_1=-\ell}\tau_{\mathbb H_\ell,p}(0,z)\geq c_2\cdot\varphi_p(\mathbb H_\ell).
\end{equation}
Using this and the definition of $L(p)$, we obtain for $k\leq j \leq L(p)-1$, 
\begin{equation}\label{eq:pf chi LB 4}
	\sum_{y: \:y_1=-(j+1-k)}\tau_{\mathbb H_{j+1-k}}(0,y) \geq c_2 \varphi_p(\mathbb H_{j+1-k})\geq \frac{c_2}{2d}\cdot\varphi_p(\Lambda_{j+1-k})\geq \frac{c_2}{2d}\cdot\frac{1}{2},
\end{equation}
where we used symmetry in the second inequality.
Choosing $\varepsilon=\tfrac{c_2}{8d}$ (which forces $K$ to be large enough), plugging \eqref{eq:pf chi LB 4} in \eqref{eq:pf chi LB 3}, and summing it over $0\leq k\leq \lfloor \tfrac{L(p)}{2}\rfloor$ yields
		\begin{equation}\label{lower.chi(p)1}
			\chi(p)\geq c_3 \cdot \inf_{j<L(p)}\mathbb E_p\big[|\widetilde {\mathcal X}_j^K|\big]  \cdot L(p)^2,
		\end{equation}
		where $c_3>0$ (and depends on $K,d$ only). In Lemma \ref{lem: bound mathcal X tilde} below, we prove that for $K$ even larger, there exists $c_4=c_4(K,d)>0$ such that
		\begin{equation}\label{eq:pf chi LB 5}
			\inf_{j<L(p)}\mathbb E_p\big[|\widetilde {\mathcal X}_j^K|\big]\geq c_4.
		\end{equation}
Combining \eqref{eq:pf chi LB 4} and \eqref{eq:pf chi LB 5} concludes the proof.		 
\end{proof}

We now turn to the proof of Lemma \ref{lem: bound mathcal X tilde}. Observe that a similar result was derived in Proposition \ref{prop: in average pioneers are regular and escapable} for the sets $\mathcal X_S^K$ with $S\subset \Lambda_{L(p)-1}$: letting $c_0=c_0(K,d)>0$ be given by the same result, and using the definition of $L(p)$,
\begin{equation}
	\mathbb E_p\big[|\mathcal X_S^K|\big]\geq c_0\cdot \mathbb E_p[\mathcal P_S]\gtrsim c_2\cdot \varphi_p(S)\geq \frac{c_0}{2}.
\end{equation}
\begin{Lem}\label{lem: bound mathcal X tilde} There exists $K_5>0$ such that, for every $K\geq K_5$, there exists $c=c(K,d)>0$ such that
\begin{equation}
				\inf_{j<L(p)}\mathbb E_p\big[|\widetilde {\mathcal X}_j^K|\big]\geq c.
\end{equation}
\end{Lem}
To prove Lemma \ref{lem: bound mathcal X tilde}, we rely on the following result from \cite{asselah2024intersection}.  
\begin{Thm}[\hspace{1pt}{\cite[Corollary~1.4]{asselah2024intersection}}]\label{thm: asselah schapira} Let $d>6$ and assume that the upper bound in \eqref{eq: 2pt full space estimate} holds. There exists $c,C>0$ such that, for every finite $V\subset \mathbb Z^d$ and every $t>0$,
\begin{equation}
	\mathbb P_{p_c}\Big[|\mathcal C(0)\cap V|\geq t|V|^{4/d}\Big]\leq C\exp(-ct).
\end{equation}
\end{Thm}

\begin{Rem} Theorem \ref{thm: asselah schapira} is a short and elementary consequence of the tree-graph inequalities of Aizenman and Newman \cite{AizenmanNewmanTreeGraphInequalities1984}.
\end{Rem}
\begin{proof}[Proof of Lemma~\textup{\ref{lem: bound mathcal X tilde}}] We make small modifications to the arguments used to derive Proposition \ref{prop: in average pioneers are regular and escapable}, i.e.\ we obtain the appropriate versions of Lemmas \ref{lem: regular are in average also escapable } and \ref{lem: pioneers are in average regular}. Below we fix $0\leq j<L(p)$. The proof is split into three steps. 

\vspace{5pt}

$\bullet$ First, we remove the restriction on the $K$-escapability. Using the exact same argument\footnote{Importing the notations around \eqref{eq:comparing cluster omega to cluster phi(omega)}, we rely on the fact that, if $K$ is large enough, for $s\geq K$ and $y\in \Lambda_K(u)$ with $y_1=-(j+1)$,
\begin{equation*}
	|\mathcal C_{\varphi(\omega)}(u)\cap \mathcal V_s^j(y)|\leq |\mathcal C_{\omega}(u)\cap \mathcal V_{2s}^j(y)|+dK\leq (2s)^{4-\tfrac{1}{d}}+dK\leq 20s^{4-\tfrac{1}{d}}.
\end{equation*}
} as the one used in Lemma \ref{lem: regular are in average also escapable }, we obtain that for every $K$ large enough, there exists $c_1=c_1(K,d)>0$ such that,
\begin{equation}\label{eq:pf lem lb 1}
	\mathbb E_p\big[|\widetilde{\mathcal X}^K_j|\big]\geq c_1 \mathbb E_p\big[\widetilde{\mathcal P}_{j}^K\big],
\end{equation}
where $\widetilde{\mathcal P}_{j}^K$ is the cardinality of the set of pioneers of $\mathbb H_j$ that are $(K,1)$-regular (in $\mathbb Z^d$) and which satisfy 
\begin{equation}\label{eq:pf lem lb 2}
|\mathcal C(u) \cap \mathcal V_s^j(y)| \le s^{4-\frac 1d}, \quad \forall s\ge K, \ \forall y\in \Lambda_K(u), \: y_1=-(j+1).
\end{equation}
We stress that this is the place where we use that the local surgery of the proof of Lemma \ref{lem: regular are in average also escapable } is sufficiently robust, see Remark \ref{rem:comparing mathcal X to pioneers}.

\vspace{5pt}

$\bullet$ Second, we observe that, by symmetry,
\begin{equation}\label{eq:pf lem lb 1.4}
	\mathbb E_p\big[\widetilde{\mathcal P}_j^K\big]\geq \frac{1}{2d}\mathbb E_p\big[\widetilde{Y}^K_{j}\big],
\end{equation}
where $\widetilde{Y}_j^K$ is the cardinality of the set of pioneers of $\Lambda_j$ that are $(K,1)$-regular, and which satisfy $u_1=-j$ (or $u_1\in \{-j,\ldots,-j+L-1\}$ in the spread-out case) and also \eqref{eq:pf lem lb 2}.

\vspace{5pt}

$\bullet$ Finally, we prove that for $K$ potentially even larger,
\begin{equation}\label{eq:pf lem lb 3}
	\mathbb E_p\big[\widetilde{Y}^K_j\big]\geq \frac{1}{4d}\mathbb E_p\big[\mathcal P_{\Lambda_j}\big].
\end{equation}
This is enough to conclude the proof. Indeed, combining \eqref{eq:pf lem lb 1}, \eqref{eq:pf lem lb 1.4}, and \eqref{eq:pf lem lb 3} gives, for $K$ large enough,
\begin{equation}
	\mathbb E_p\big[|\widetilde{\mathcal X}^K_j|\big]\geq c_1 \mathbb E_p\big[\widetilde{\mathcal P}_{j}^K\big]\geq \frac{c_1}{8d^2}\mathbb E_p[\mathcal P_{\Lambda_j}]\gtrsim  c_1\varphi_p(\Lambda_j)\geq  \frac{c_1}{2},
\end{equation}
where the third inequality follows from \eqref{eq:comparison pioneers and phi}, and the fourth one follows from the assumption $j<L(p)$.
To prove \eqref{eq:pf lem lb 3}, we reproduce the argument of Lemma \ref{lem: pioneers are in average regular}. Recall the definition of $\mathcal T_s^{\textup{loc}}(u)$ from \eqref{eq:def ts loc}. The main difference is that a point $u$ is now called \emph{$s$-locally super bad} if $\mathcal T_s^{\textup{loc}}(u)$ does not occur \textbf{or} if $\mathcal T_s^{\textup{loc}}(u)$ occurs but there exists $y\in \Lambda_K(u)$ and $z\in \Lambda_{s^{2d}}(u)$ such that
\begin{equation}
	|\mathcal C(z;\Lambda_{s^{2d}}(u))\cap \mathcal V_s^j(y)|>s^{4-\frac{1}{d}}(\log s)^{-3}.
\end{equation}
As in Lemma \ref{lem:Tsloc.Ts}, if for every $s\geq K$, $u$ is not $s$-locally super bad, then, $u$ is $(K,1)$-regular and it satisfies \eqref{eq:pf lem lb 2}.  As a result,
\begin{equation}
	\mathbb E_p\big[\widetilde{Y}^K_j\big]\geq \frac{1}{2d}\mathbb E_p[\mathcal{P}_{\Lambda_j}]-\sum_{s\geq K}\mathbb E_p\big[\mathcal{P}_{\Lambda_j}^{s\textup{-loc-super-bad}}\big],
\end{equation}
where we used symmetries of the model to get that (recall that $L=1$ in the nearest-neighbour case) 
\begin{equation}
\sum_{\substack{u\in \partial \Lambda_j\\u_1\in \{-j,\ldots,-j+L-1\}}}\tau_{\Lambda_j,p}(0,u)\geq \frac{1}{2d}\mathbb E_p[\mathcal{P}_{\Lambda_j}],
\end{equation}
and where we wrote $\mathcal{P}_{\Lambda_j}^{s\textup{-loc-super-bad}}$ for the number of $s$-locally super bad pioneers of $\Lambda_j$.
The proof of Lemma \ref{lem: pioneers are in average regular} adapts once we prove an upper bound of the form $e^{-c(\log s)^4}$ on the probability of $u$ being $s$-locally super-bad. By Lemma \ref{lem:Tsloc}, there exists $c_2=c_2(d)>0$, such that
\begin{align}
	\mathbb P_p[u \textup{ is}& \textup{ $s$-locally super bad}]\notag
	\\&\leq e^{-c_2(\log s)^4}+|\Lambda_{s^{2d}}(u)|^2\cdot \max_{\substack{y\in \Lambda_K(u)\\z\in \Lambda_{s^{2d}}(u)}}\mathbb P_p\Big[|\mathcal C(z;\Lambda_{s^{2d}}(u))\cap \mathcal V_s^j(y)|>s^{4-\frac{1}{d}}(\log s)^{-3}\Big].\label{eq:pf lem lb 4}
\end{align}
We fix $y\in \Lambda_K(u)$ and $z\in \Lambda_{s^{2d}}(u)$, and analyse the last probability in \eqref{eq:pf lem lb 4} using Theorem \ref{thm: asselah schapira}. There exist $c_3,C_1>0$ (which only depend on $d$) such that, for any finite set $V$ and any $t>0$, 
	\begin{equation}\label{eq:pf lem lb 5}
	\mathbb P_{p_c}\Big[|\mathcal C(z)\cap V|\ge t |V|^{4/d}\Big] \le C_1\exp(-c_3t). 
	\end{equation}
	To use this result in our context, we just need an upper bound on $|\mathcal V_s^j(y)|$. By definition of $\mathcal V_s^j(y)$,   
	\begin{equation} 
	\sum_{b \in \mathcal V_s^j(y)} \tau_{\mathbb H_{j+1},p}(y,b) \ge s^{ \frac 32-d} \cdot |\mathcal V_s^j(y)|, 
	\end{equation}
	and by Theorem~\ref{thm:main} and translation invariance, 
		\begin{equation} 
	\sum_{b \in \mathcal V_s^j(y)} \tau_{\mathbb H_{j+1},p}(y,b) \le \sum_{b \in \Lambda_s(y)} \tau_{\mathbb H_{j+1},p}(y,b)\lesssim \sum_{k=0}^s \varphi_p(\mathbb H_k) \lesssim 1+s. 
	\end{equation}
	Combining the last two displays yields the deterministic bound $|\mathcal V_s^j(y)|\leq C_2 s^{d-1/2}$ for some $C_2=C_2(d)>0$. Applying \eqref{eq:pf lem lb 5} to $V=\mathcal V^j_s(y)$, and letting $t_0=C_2^{-4/d}s^{1/d}(\log s)^{-3}$ gives
	\begin{equation}\label{eq:pf lem lb 6}
	\mathbb P_{p_c}\Big[|\mathcal C(z;\Lambda_{s^{2d}}(u))\cap \mathcal V_s^j(y)|\geq s^{4-\frac 1d}(\log s)^{-3}\Big]\leq \mathbb P_{p_c}\Big[|\mathcal C(z)\cap \mathcal V_s^j(y)|\geq t_0|\mathcal V_s^j(y)|^{4/d}\Big]\leq e^{-c_3 t_0}.
	\end{equation}
	Combining \eqref{eq:pf lem lb 4} and \eqref{eq:pf lem lb 6} gives $c_4,c_5>0$ (which only depend on $d$) such that,
	\begin{equation}
		\mathbb P_p[u \textup{ is} \textup{ $s$-locally super bad}]\lesssim e^{-c_2(\log s)^4}+ (s^{2d^2})^2 e^{-c_4 s^{1/d}(\log s)^{-3}}\lesssim e^{-c_5(\log s)^4}.
	\end{equation}
This concludes the proof.
\end{proof}

The upper bounds in Theorems~\ref{thm:nc bounds phi(S)} and~\ref{thm:bounds 2pt} 
follow easily by combining Lemma~\ref{lem:SL} and Theorems~\ref{thm:main} and \ref{thm:sharplength}.

\begin{proof}[Proof of the upper bound in Theorem~\textup{\ref{thm:nc bounds phi(S)}}] Let $p\in [p_c/2,p_c)$. By Theorem \ref{thm:sharplength}, it suffices to prove the existence of $c_1,C_1>0$ (which only depend on $d$) such that, for every $S\subset \mathbb Z^d$ containing $0$,
\begin{equation}
	\varphi_p(S)\leq C_1\exp\Big(-c_1\frac{\mathrm{d}(0,\partial S)}{L(p)}\Big).
\end{equation}
 Fix $S\subset \mathbb Z^d$ containing $0$. We may assume that $r:=\textup{d}(0,\partial S)\geq 10(L(p)+L)$ (where $L=1$ in the nearest-neighbour case), otherwise the result follows from Theorem \ref{thm:main}. Fix $S_0\subset \Lambda_{L(p)}$ such that $\varphi_p(S_0)<\tfrac{1}{2}$. Applying Lemma \ref{lem:SL} to $x\in \partial S$, multiplying by $p$, and summing over $y\notin S$ with $x\sim y$, and then over $x$ gives
 
 \begin{equation}
\varphi_p(S) \leq \sum_{\substack{x\in S\\y\notin S\\ x\sim y}}  
\sum_{\substack{u_1\in S_0\\ v_1\notin S_0 \\ u_1\sim v_1}} \tau_{S_0,p}(0,u_1)\cdot p \cdot  
\tau_{S,p}(v_1,x) \cdot p.
\end{equation}
 Iterating this procedure $k=\lfloor \frac{r}{2(L(p)+L)}\rfloor$ times with translates of $S_0$ gives
 \begin{multline}\label{eq:iterated-phi}
	\varphi_p(S)\leq  \sum_{\substack{x\in S\\y\notin S\\ x\sim y}}   
	\sum_{\substack{u_1\in S_0\\ v_1\notin S_0\\ u_1\sim v_1}} \tau_{S_0,p}(0,u_1)\cdot p
	\ldots\sum_{\substack{u_{k}\in (S_0+v_{k-1})\\ v_{k}\notin (S_0+v_{k-1}) \\ u_k\sim v_k}} \tau_{S_0+v_{k-1},p}(v_{k-1},u_k)\cdot p\cdot \tau_{S,p}(v_k,x)\cdot p
	\\\leq \varphi_p(S_0)^k \cdot \sup_{z\in S}\varphi_p(S-z)\lesssim 2^{-k}\end{multline} 
where the last inequality follows from Theorem \ref{thm:main}. This concludes the proof.
\end{proof}

\begin{proof}[Proof of the upper bound in Theorem~\textup{\ref{thm:bounds 2pt}}] Let $p\in [p_c/2,p_c)$. Again, by Theorem \ref{thm:sharplength}, it suffices to prove the existence of $c_1,C_1>0$ (which only depend on $d$) such that, for every $x\in \mathbb Z^d$,
\begin{equation}
	\tau_p(0,x)\leq \frac{C_1}{(1\vee |x|)^{d-2}}\exp\Big(-c_1\frac{|x|}{L(p)}\Big).
\end{equation}
We may (again) assume that $|x|\geq 10(L(p)+L)$, otherwise the result follows from $\tau_p(0,x)\leq \tau_{p_c}(0,x)$ and \eqref{eq: 2pt full space estimate}. Fix $S_0\subset \Lambda_{L(p)}$ such that $\varphi_p(S_0)<\tfrac{1}{2}$. Iterating Lemma \ref{lem:SL} $k=\lfloor \frac{|x|}{2(L(p)+L)}\rfloor$ times, exactly as above, gives
\begin{equation}
	\tau_p(0,x)\leq \varphi_p(S_0)^k\sup_{z\in \Lambda_{kL(p)}}\tau_p(z,x)\lesssim \frac{2^{-k}}{|x|^{d-2}},
\end{equation}
where in the last inequality we used \eqref{eq: 2pt full space estimate} and the fact that (by definition of $k$) one has $kL(p)\leq |x|/2$. This concludes the proof.
\end{proof}

\section{Derivation of the one-arm exponent}\label{sec:one arm}

We now turn to the proof of Theorem \ref{thm:onearm}. It follows from a combination of Theorems \ref{thm:sharplength} and \ref{thm:nc bounds phi(S)} and Corollary \ref{coro:bounded derivative}. We begin with a lemma, which provides a proof of \eqref{eq:theta drops intro}. Recall that $\theta_n(p)=\mathbb P_p[0\connect{}\partial \Lambda_n]$.
\begin{Lem}\label{lem:theta drops} Let $d>6$ and assume that \eqref{eq: 2pt full space estimate} holds. There exists $K>0$ such that, for every $n$ large enough, 
\begin{equation}
	\theta_n(p_c-K/n^2)\leq \frac{1}{10}\cdot \theta_{\lfloor n/2\rfloor}(p_c).
\end{equation}
\end{Lem}
\begin{proof} Set $B:=\Lambda_{\lfloor n/2\rfloor-L}$ (where $L$ is the spread parameter of the model, equal to $1$ in the nearest-neighbour case). Then,
\begin{equation}
	\{0\connect{}\partial \Lambda_n\}\subset \bigcup_{\substack{u\in B\\ v\notin B\\u \sim v}}\{0\connect{B\:}u\}\circ \{\{u,v\}\textup{ is open}\}\circ \{v\connect{} \partial \Lambda_n\}.
\end{equation}
Applying the BK inequality \eqref{eq:BK ineq}, gives, for $p\leq p_c$,
\begin{equation}
	\theta_n(p)\leq \varphi_p(B)\cdot \theta_{\lfloor n/2\rfloor}(p_c).
\end{equation}
Now, using Theorem \ref{thm:nc bounds phi(S)} with $p=p_c-K/n^2$ gives the existence of $c,C>0$ such that, for every $n$ large enough,
\begin{equation}
	\varphi_p(B)\leq C\exp(-cK^{1/2}).
\end{equation}
Choosing $K$ large enough so that $\varphi_p(B)\leq \tfrac{1}{10}$ concludes the proof.
\end{proof}

\begin{proof}[Proof of Theorem~\textup{\ref{thm:onearm}}] We first prove the upper bound in Theorem \ref{thm:onearm}. By Corollary \ref{coro:bounded derivative}, there exists $C_1>0$ such that, for every $p\in (p_c/2,p_c]$, and every $n\geq 1$,
\begin{equation}\label{eq:pf 1arm 1}
	\frac{\mathrm{d}}{\mathrm{d}p}\theta_n(p)\leq C_1.
\end{equation}
Let $K>0$ be given by Lemma \ref{lem:theta drops}. 
Integrating \eqref{eq:pf 1arm 1} between $p:=p_c-K/n^2$ and $p_c$ yields, for $n$ large enough (so that $p>p_c/2$),
\begin{equation}\label{eq:pf 1arm 2}
	\theta_n(p_c)\leq \theta_n(p)+\frac{KC_1}{n^2}.
\end{equation}
Using Lemma \ref{lem:theta drops} in \eqref{eq:pf 1arm 2} yields the existence of $n_0\geq 1$ such that, for every $n\geq n_0$, \begin{equation}\label{eq:pf 1arm 3}
	\theta_n(p_c)\leq \frac{1}{10}\cdot \theta_{\lfloor n/2\rfloor}(p_c)+\frac{KC_1}{n^2}.
\end{equation}
From \eqref{eq:pf 1arm 3}, we prove by induction that $\theta_{2^k}(p_c)\leq \frac{C_2}{(2^k)^2}$ for every $k\geq 0$, for a well-chosen $C_2$. First, if $C_2\geq n_0^2$, we get that the bound holds for every $k\leq \log_2(n_0)$. We assume that the bound holds at $2^k\geq n_0$ and prove that it also holds at $n=2^{k+1}$. Applying \eqref{eq:pf 1arm 3} to this value of $n$ and using the induction hypothesis at $n/2$ gives
\begin{equation}
	\theta_{n}(p_c)\leq \frac{1}{10}\cdot \frac{C_2}{(n/2)^2}+\frac{KC_1}{n^2}.
\end{equation}
Setting $C_2:=n_0^2\vee \tfrac{5}{2}KC_1$ gives $\theta_n(p_c)\leq \tfrac{C_2}{n^2}$, which concludes the induction. It is routine to deduce the bound on $\theta_n(p_c)$ for every $n\geq 1$ from there.

We now turn to the lower bound in Theorem \ref{thm:onearm}. We follow the argument of \cite{vEGPSperco}. An alternative proof can be obtained through a direct second moment method \cite{sakai2004mean,KozmaNachmias}. By Theorem \ref{thm:sharplength}, there exists $c_1>0$ such that for every $p\in[p_c/2,p_c)$, one has $L(p)\geq c_1(p_c-p)^{-1/2}$. Thus, if $p_c-\frac{c_1^{2}}{n^2}< p \leq p_c$, one has $L(p)>n$ so that by \eqref{eq:derivative of theta}, 
\begin{equation}\label{eq:proof onearm 4}
	\frac{\mathrm{d}}{{\mathrm{d}}p}\theta_n(p)\geq \frac{1-\theta_n(p)}{2p(1-p)}.
\end{equation}
Integrating \eqref{eq:proof onearm 4} on the same interval of $p$ and using that $\theta_n(p_c)\rightarrow 0$ gives $c_2>0$ such that, for every $n$ large enough,
\begin{equation}
	-\log(1-\theta_n(p_c))\asymp \theta_n(p_c)\geq \frac{c_2}{n^2}.
\end{equation} 
This concludes the proof.
\end{proof}

\section{Near-critical lower bounds}\label{sec:nc lb}

In this last section, we establish the lower bounds in 
Theorems~\ref{thm:nc bounds phi(S)} and~\ref{thm:bounds 2pt}. The proofs follow the natural strategy of reversing the argument used to 
establish the upper bounds, replacing Lemma~\ref{lem:SL} with 
Theorem~\ref{thm:pre reversed SL}. The section is 
divided into two parts: the first concerns averaged lower bounds, and the second 
pointwise lower bounds. For simplicity, we write the arguments in the nearest-neighbour case. The extension to the spread-out case is straightforward. Recall that we have assumed that $d>6$ and that \eqref{eq: 2pt full space estimate} holds.

\subsection{Averaged near-critical lower bounds}
The goal of this subsection is to prove the lower bound in Theorem \ref{thm:nc bounds phi(S)}. We will deduce it from Corollary \ref{cor:partial monot} and a near-critical lower bound on the $\varphi_p(\Lambda_n)$ for $n\geq 1$. In fact, we will prove the following stronger result. 

We begin with some notations. For $1\leq i\leq d$, recall that $\mathbf{e}_i$ denotes the vector of the canonical basis of $i$-th coordinate equal to $1$. If $k\ge 1$, the \emph{facets} of the box $\Lambda_k$ are defined to be the $2d$ sets $F_i^\pm(k) := \Lambda_{\lfloor k/2\rfloor }(\pm k\mathbf e_i)\cap \partial \Lambda_k$, where 
$i\in \{1,\dots,d\}$. By extension, given $z\in \mathbb Z^d$, we call facets of the box $\Lambda_k(z)$, the sets $z+ F_i^\pm(k)$, $i\in \{1,\dots,d\}$. 
\begin{Prop}\label{prop: nc lb averaged} Let $\varepsilon\in (0,1)$. There exist $\alpha>0$ and $p_0\in (0,p_c)$ which only depend on $\varepsilon$ and $d$ such that the following holds for $p\in [p_0,p_c)$: setting $k:=\lfloor \varepsilon L(p)\rfloor\geq 1$, for every $N\ge 1$, every sequence $(x_j)_{1\le j\le N}$ with $x_1= 0$ and $x_{j+1} - x_j \in \{2k\mathbf e_1,\dots,2k\mathbf e_d\}$ for every $j\le N-1$, and any facet $F$ of $\Lambda_k(x_N)$, one has, with $Q:=\cup_{1\le j\le N} \Lambda_k(x_j)$, 
\begin{equation}
\sum_{y\in F} \tau_{Q,p} (0,y) \ge \alpha^N. 
\end{equation}
\end{Prop}

We are now in a position to prove the lower bound in Theorem \ref{thm:nc bounds phi(S)}.

\begin{proof}[Proof of the lower bound in Theorem~\textup{\ref{thm:nc bounds phi(S)}}] Let $C_0=C_0(d)>0$ be given by Corollary \ref{cor:partial monot}. Let $p<p_c$ to be chosen close enough to $p_c$. The result is clear if $S$ is infinite. We therefore assume that $S$ is finite and let $\ell$ be the smallest integer such that $S\subset \Lambda_\ell$. Observe that $\ell\leq \mathrm{diam}(S)\leq 2\ell$. We assume that $\ell\geq L(p)$, otherwise the bound $\varphi_p(S)\geq \tfrac{1}{2}$ holds by definition of $L(p)$. By Corollary \ref{cor:partial monot}, one has
\begin{equation}
	\varphi_p(S)\geq C_0^{-1}\cdot \varphi_p(\Lambda_\ell),
\end{equation}
so it suffices to find a lower bound on $\varphi_p(\Lambda_\ell)$. Set $k:=\lfloor \tfrac{L(p)}{4}\rfloor$, $N:=\lceil \tfrac{\ell}{k}\rceil$. Let $\alpha>0$ and $p_0\in (0,p_c)$ be given by Proposition \ref{prop: nc lb averaged} (with $\varepsilon=\tfrac{1}{4}$). Defining the sequence $(x_i)_{1\leq i \leq N}$ as: $x_1=0$ and $x_{i+1}=x_i+2k\mathbf{e}_1$, we find that (by definition) $F:=x_N+F^{+}_1(k)$ lies on the boundary of $\Lambda_{2(N-1)k+k}$. Observe that $\ell \leq 2(N-1)k+k\leq 4\ell$. Letting $Q:=\cup_{1\leq j\leq N}\Lambda_k(x_j)$, a new application of Corollary \ref{cor:partial monot} and Proposition \ref{prop: nc lb averaged} give
\begin{equation}
	\varphi_p(\Lambda_\ell)\geq C_0^{-1}\varphi_p(\Lambda_{2(N-1)k+k})\gtrsim \sum_{y\in F}\tau_{Q,p}(0,y)\geq \alpha^N\geq c_1\exp\Big(-C_1\frac{\textup{diam}(S)}{L(p)}\Big),
\end{equation}
for some $c_1,C_1>0$ (which depend only on $d$). Combining the above displayed equations concludes the proof.
\end{proof}

We turn to the proof of Proposition \ref{prop: nc lb averaged}. We will prove this result by induction over $N\geq 1$. Before moving to the proof, we start with a preparatory lemma. 

\begin{Lem}\label{lem:facet}
There exist $c,k_0>0$, such that for every $k_0\leq k<L(p)$, every facet $F$ of the box $\Lambda_k$, one has 
\begin{equation}\label{eq:facet1}
\sum_{y\in F} \tau_{\Lambda_k,p}(0,y) \ge c, 
\end{equation}
and
\begin{equation}\label{eq:facet2}
\inf_{z \in \Lambda_{\lfloor k/2\rfloor}(-k\mathbf e_1) } \sum_{y\in F} \tau_{Q_k,p}(z,y) \ge c,
\end{equation}
where $Q_k := \Lambda_k \cup \Lambda_k(-2k\mathbf e_1)$.
\end{Lem}
\begin{proof}For simplicity, we assume that $k$ is even. The proof in the odd case follows the same route with only minor modifications. 
\vspace{5pt}

Let us start with the proof of~\eqref{eq:facet1}. For $1\le j<L(p)$, define $U_j := \{z \in \partial \Lambda_j : z_1 = j, \ z_i\ge 0 \text { for all }i\ge 2\}$. Since $\partial \Lambda_j$ is the union of $d 2^d$ copies of $U_j$, one has, using symmetry and the definition of $L(p)$, 
\begin{equation}\label{connect face 0}
\sum_{y\in U_j} \tau_{\Lambda_j,p}(0,y) \ge \frac{1}{d2^d}\frac{1}{2d p_c}\varphi_p(\Lambda_j)\ge \rho :=\frac {1}{2^{d+1}d^2p_c}.
\end{equation}
Now, fix $k<L(p)$ to be chosen large enough below, and a facet $F=F_i^{\pm}(k)$ of $\Lambda _k$. Let $H  := \Lambda_k\cap \Lambda_{k/2}(\pm\mathbf{e}_i)$, see Figure \ref{fig: lemma averaged lb} for an illustration. Note that for any $z \in H$, letting $j = \mathrm{d}(z,F)$, there exists a copy of $U_j$ on $\Lambda_j(z)$, which lies entirely in $F$. 
Denoting by $U_j(z,F)$ this copy of $U_j$, one has by~\eqref{connect face 0}, 
\begin{equation}\label{lemforH}
\sum_{y\in F} \tau_{\Lambda_k,p}(z,y) \ge \sum_{y\in U_j(z,F)} \tau_{\Lambda_j(z),p} (z,y) \ge \rho.  
\end{equation}
Observe next that one of the copies of $U_{k/2}$ lies entirely in $H$, and in fact by symmetry one can always assume that $U_{k/2}\subset H$. Let $K\geq 1$ to be chosen large enough and assume that $k\geq 4K$. We now use Proposition \ref{prop: effective reversed SL} to argue that there exists $c_1=c_1(K,d)>0$ such that 
\begin{align}
	\sum_{y\in F}\tau_{\Lambda_k,p}(0,y)&\geq c_1\sum_{y\in F}\sum_{\substack{u\in U_{k/2}}}\mathbb E_p\Big[\mathds{1}\{u\in \mathcal X^K_{\Lambda_{k/2},\Lambda_k}\}\Big(\tau_{\Lambda_k,p}(\mathbf{w},y)-\mathsf{E}_{\Lambda_k}(\mathbf{w},y;\mathcal C(u;\Lambda_k);p)\Big)\Big],\label{connect face 1}
\end{align}
where we used that $U_{k/2}\subset \partial \Lambda_{k/2}$. As mentioned in Remark \ref{rem:comparing mathcal X to pioneers} and since $k\geq 4K$, we may assume that for $u\in \mathcal X^K_{\Lambda_{k/2},\Lambda_k}\cap U_{k/2}$, one has $\mathbf{w}=u+\lfloor \tfrac{K}{2}\rfloor\mathbf{e}_1$.

By construction, the $\mathbf{w}$ given on the right-hand side of \eqref{connect face 1} almost surely satisfies $\mathbf{w}\in H$. Thanks to \eqref{lemforH} we obtain
\begin{equation}\label{connect face 2}
\sum_{y\in F}  \tau_{\Lambda_k,p}(\mathbf{w},y) \ge \rho.
\end{equation}
Moreover, for a $u$-good set $A$, one has $C_1=C_1(d)>0$ such that
\begin{equation}\label{connect face 2.5}
	\sum_{y\in F}\mathsf{E}_{\Lambda_k}(\mathbf{w},y;A;p)\lesssim \sum_{a\in A}\tau_{p_c}(\mathbf{w},a)\sum_{y\in F}\tau_{\Lambda_k,p_c}(a,y)\lesssim \sup_{n\geq 0}\varphi_{p_c}(\mathbb H_n) \cdot \sum_{a\in A}\tau_{p_c}(\mathbf{w},a)\leq \frac{C_1}{\sqrt{K}},
\end{equation}
where in the last inequality we used Theorem \ref{thm:main} and the same computation as in \eqref{eq:pf main rep5}. We now choose $K$ large enough so that $C_1/\sqrt{K}\leq \rho/2$. We deduce from all the above that 
\begin{equation}
\sum_{y \in F} \tau_{\Lambda_k,p}(0,y) \ge  \frac{c_1 \rho}{2} \cdot \mathbb E_p\big[|\mathcal X_{\Lambda_{k/2},\Lambda_k}^K\cap U_{k/2}|\big]\geq \frac{c_1 \rho}{2} \cdot \mathbb E_p\big[|\mathcal X_{\Lambda_{k/2}}^K\cap U_{k/2}|\big],
\end{equation}
where in the second inequality we used that (since $k\geq 4K$) a point $u\in U_{k/2}$ that is $K$-escapable is automatically $K$-escapable in $\Lambda_k$.
Using Proposition \ref{prop: in average pioneers are regular and escapable}, the fact that $k<L(p)$, and the same symmetry argument as in \eqref{connect face 0} yields that for $K$ even larger there exists $c_2=c_2(K,d)>0$, such that 
\begin{equation}\label{connect face 2.5 bis}
\mathbb E_p\big[|\mathcal X_{\Lambda_{k/2}}^K\cap U_{k/2}|\big]\ge c_2\rho.
\end{equation} 
Combining the last two displays concludes the proof of~\eqref{eq:facet1}. 
\vspace{5pt}

We now turn to the proof of \eqref{eq:facet2}. First, observe that if $F=F_{1}^-(k)$ then the result is immediate. Indeed, if $z\in \Lambda_{k/2}(-k\mathbf{e}_1)$ satisfies $z_1=-k+j$ with $j\in \{-k/2,\ldots,k/2\}$, there exists a translate of $U_j$ that entirely lies in $F$ so that the result follows by a similar computation as in \eqref{lemforH}. We therefore assume that $F\neq F^-_1(k)$. For simplicity, we even assume that $F=F^+_1(k)$. The argument for the general case follows a similar route, and we illustrate it in Figure \ref{fig: lemma averaged lb}. We define a sequence $(H_i)_{1\leq i \leq 5}$ as follows: $H_1=H$ and $H_{i+1}=H_i-\tfrac{k}{2}\mathbf{e}_1$. We prove by induction over $1\leq i \leq 5$ that,
\begin{equation}\label{eq: connect face 3}
	\inf_{z\in H_i}\sum_{y\in F}\tau_{Q_k,p}(z,y)\geq c_3^i,
\end{equation}
for some well-chosen $c_3=c_3(d)>0$. This is enough to get \eqref{eq:facet2} since $\Lambda_{k/2}(-k\mathbf{e}_1)\subset \cup_{1\leq i \leq 5}H_i$.  Observe that the case $i=1$ follows from \eqref{lemforH} if $c_3\leq \rho$, which we now assume. Let $2\leq i \leq 5$. Suppose that \eqref{eq: connect face 3} holds for $i-1$. We prove it for $i$. Let $z\in H_i$. As above, there exists a copy of $U_{k/2}$ which lies on $\Lambda_{k/2}(z)\cap H_{i-1}$. We call $U$ this copy. Using Proposition \ref{prop: effective reversed SL} as before, we get, for $K\geq 1$,
\begin{equation}
	\sum_{y\in F}\tau_{Q_k,p}(z,y)\geq c_1 \sum_{y\in F}\sum_{\substack{u\in U}}\mathbb E_p\Big[\mathds{1}\{u\in \mathcal X^K_{\Lambda_{k/2}(z),Q_k}\}\Big(\tau_{Q_k,p}(\mathbf{w},y)-\mathsf{E}_{Q_k}(\mathbf{w},y;\mathcal C(u;Q_k);p)\Big)\Big],
\end{equation}
where we assume (again) that, for $u\in \mathcal X_{\Lambda_{k/2}(z),Q_k}^K$, the $\mathbf{w}$ given by the $K$-escapability of $u$ in $Q_k$ satisfies $\mathbf{w}=u+\lfloor\tfrac{K}{2}\rfloor\mathbf{e}_1$. In particular, $\mathbf{w}\in H_{i-1}$.
By the induction hypothesis,  
\begin{equation}
	\sum_{y\in F}\tau_{Q_k,p}(\mathbf{w},y)\geq c_3^{i-1}.
\end{equation}
Moreover, the same computation as in \eqref{connect face 2.5} gives that for $K$ large enough and a $u$-good set $A$,
\begin{equation}
	\sum_{y\in F}\mathsf{E}_{Q_k}(\mathbf{w},y;A;p)\leq \frac{c_3^{4}}{2}.
\end{equation}
Combining the three above displays and \eqref{connect face 2.5 bis} gives
\begin{equation}
	\sum_{y\in F}\tau_{Q_k,p}(z,y)\geq c_3^{i-1}\cdot \frac{c_1}{2}\cdot \mathbb E_p\big[|\mathcal X_{\Lambda_{k/2}(z),Q_k}^K\cap U|\big]\geq c_3^{i-1}\cdot\frac{c_1c_2\rho}{2}.
\end{equation}
The induction step follows from choosing $c_3=\tfrac{c_1c_2\rho}{2}$.
\end{proof}

\begin{figure}[htb]
	\begin{center}
		\includegraphics[scale=1.1]{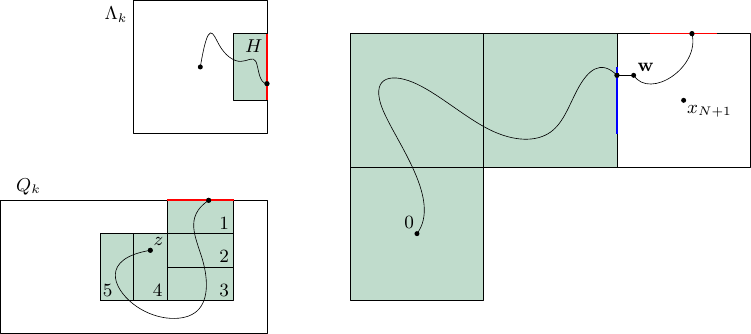}
		\caption{\textsc{Left}: An illustration of the proof of Lemma \ref{lem:facet}. The facet $F$ is the bold red line. On the bottom, we illustrated the sequence $(H_i)_{1\leq i \leq 5}$ in the case $F\neq F_1^+(k)$. \textsc{Right}: An illustration of the proof of Proposition \ref{prop: nc lb averaged}. The set $Q'$ is the light green region. The facet $F$ (resp. $F'$) is the bold red (resp. blue) line.}
		\label{fig: lemma averaged lb}
	\end{center}
\end{figure}

We are now in a position to prove Proposition \ref{prop: nc lb averaged}. We will consider restricted notions of regular and escapable points. 

\begin{proof}[Proof of Proposition~\textup{\ref{prop: nc lb averaged}}] Let $c,k_0>0$ be given by Lemma \ref{lem:facet}. Let $\varepsilon>0$. Let $\alpha,p_0$ to be fixed and $p\in [p_0,p_c)$. Define $k$ as in the statement.
 
We prove the result by induction on $N\geq 1$. Observe that the base case $N=1$ is exactly~\eqref{eq:facet1} if we ask $\alpha\leq c$ and $p$ sufficiently close to $p_c$ so that $k\geq k_0$. We therefore assume that the result holds for $N$ and consider a sequence of $N+1$ points as in the statement, as well as $F$ a facet of $\Lambda_k(x_{N+1})$. Let $Q' = \cup_{1\le j\le N} \Lambda_k(x_j) $ and $Q =\cup_{1\leq j\le N+1} \Lambda_k(x_j)$, see Figure \ref{fig: lemma averaged lb} for an illustration. If $F$ is already part of $Q'$, then the result follows from the induction hypothesis, hence we may assume that $F$ is not part of $Q'$.
 
Let $F'$ be the (unique) facet shared by $\Lambda_k(x_N)$ and $\Lambda_k(x_{N+1})$. Let $K\geq 1$. Using Proposition \ref{prop: effective reversed SL} as in the proof of Lemma \ref{lem:facet} gives $c_1=c_1(K,d)>0$ such that,
\begin{equation}
	\sum_{y\in F}\tau_{Q,p}(0,y)\geq c_1 \sum_{y\in F}\sum_{\substack{u\in F'}}\mathbb E_p\Big[\mathds{1}\{u\in \mathcal X^K_{Q',Q}\}\Big(\tau_{Q,p}(\mathbf{w},y)-\mathsf{E}_{Q}(\mathbf{w},y;\mathcal C(u;Q);p)\Big)\Big],
\end{equation}
where we (again) ask that, for $u$ as above, assuming that $F'=x_N+F^+_i(k)$, its associated $\mathbf{w}$ satisfies $\mathbf{w}=u+\lfloor\tfrac{K}{2}\rfloor\mathbf{e}_i$. By \eqref{eq:facet2}, one has
\begin{equation}
	\sum_{y\in F}\tau_{Q,p}(\mathbf{w},y)\geq \sum_{y\in F}\tau_{\Lambda_k(x_N)\cup \Lambda_k(x_{N+1}),p}(\mathbf{w},y)\geq c.
\end{equation}
Hence, reasoning as in the proof of Lemma \ref{lem:facet}, we obtain that, for $K$ large enough, there exists $c_2=c_2(K,d)>0$ such that 
\begin{equation}\label{cor:lowerbound1}
\sum_{y \in F} \tau_{Q,p}(0,y) \ge c_2\cdot \mathbb E_p\big[|\mathcal X_{Q',Q}^K \cap F'|\big]. 
\end{equation}
We will use the induction hypothesis to lower bound the expectation in \eqref{cor:lowerbound1}. We let $\mathcal X_{Q'}$ (resp. $\mathcal X_{Q',Q}^{K\textup{-reg}}$) denote the set of pioneers of $Q'$ (resp. the set of pioneers of $Q'$ that are $(K,1)$-regular in $Q$). Observe that $\mathcal X_{Q'}\supset \mathcal X_{Q',Q}^{K\textup{-reg}}\supset \mathcal X_{Q',Q}^K$. Using (a mild adaptation of) Lemma \ref{lem: regular are in average also escapable }, we find that for $K$ large enough, there exists $c_3=c_3(K,d)>0$ such that 
\begin{equation}\label{cor:lowerbound2}
\mathbb E_p\big[|\mathcal X_{Q',Q}^K \cap F'|\big] \ge c_3\cdot \mathbb E_p\big[|\mathcal X_{Q',Q}^{K\text{-reg}} \cap F'|\big]. 
\end{equation}
Hence, we need to show that most pioneers of $Q'$ lying in $F'$ are regular in $Q$. For this, we adapt the proof of Lemma \ref{lem: pioneers are in average regular}. Note that the previous argument breaks down in this setting, since 
Lemma~\ref{lem:KN} is no longer valid at scales $s>2L(p)$. 

Nevertheless, one can still make the following observation. If a point $y\in Q$ is not $(K,1)$-regular in $Q$, then one of the following must occur: there exists $K\le s \le k^{1/(2d)}$, such that $y$ is $s$-locally bad; or there exists $s>k^{1/(2d)}$, such that $|\mathcal C(y;Q) \cap \Lambda_s(y)|> s^4(\log s)^7$. If $y$ satisfies the latter property, we say that it is $s$-bad in $Q$. We may write, 
\begin{multline}\label{eq: pf averaged 1}
	\mathbb E_p\big[|\mathcal X_{Q',Q}^{K\text{-reg}} \cap F'|\big]\geq \mathbb E_p\big[|\mathcal X_{Q'}\cap F'|\big]-\sum_{K\leq s\leq k^{1/(2d)}}\mathbb E_p\big[|\mathcal{X}_{Q'}\cap \{y\in F': y\textup{ is $s$-locally bad}\}|\big]
	\\-\sum_{s> k^{1/(2d)}}\mathbb E_p\big[|\mathcal{X}_{Q'}\cap \{y\in F': y\textup{ is $s$-bad in $Q$}\}|\big]. 
\end{multline}

We first handle the first sum in \eqref{eq: pf averaged 1}.
Using the exact same argument as in the proof of Lemma \ref{lem: pioneers are in average regular}, we may prove that, for every $K$ large enough,
\begin{equation}\label{eq: pf averaged 1.5}
	\sum_{K\leq s\leq k^{1/(2d)}}\mathbb E_p\big[|\mathcal{X}_{Q'}\cap \{y\in Q: y\textup{ is $s$-locally bad}\}|\big]\leq \frac{1}{4}\cdot\mathbb E_p\big[|\mathcal X_{Q'}\cap F'|\big].
\end{equation}
This works because in this regime of $s$, Lemma \ref{lem:KN} can be applied.

We now bound the second term on the right-hand side of \eqref{eq: pf averaged 1}. If $s>k^{1/(2d)}$, we let
\begin{equation}\label{def:st}
t:= \max\big(2s ,s + L(p) (\log s)^2\big), 
\end{equation}
and 
\begin{align}
\widetilde {\mathcal T}_{s,Q}^{\textrm{loc}}(y) :=& \Big\{|\mathcal C\big(z; \Lambda_t(y)\cap Q\big)\cap \Lambda_s(y)|\leq s^4 (\log s)^4,\ \forall z \in \Lambda_s(y)\Big\} \\ 
\nonumber				&\cap \Big\{ \exists  \text{ at most }  (\log s)^3 \text{ disjoint open paths from }\partial \Lambda_s(y) \text{ to } \partial \Lambda_t(y)\text{ in }\Lambda_t(y)\cap Q\Big\}.		
\end{align}
We let $\mathcal A_1(y)$ and $\mathcal A_2(y)$ denote the two events above.
Observe that $\{y \text{ is $s$-bad in $Q$}\}\subseteq \widetilde {\mathcal T}_{s,Q}^{\textrm{loc}}(y)^c$. Our goal is now to bound $\mathbb P_p[\widetilde {\mathcal T}_{s,Q}^{\textrm{loc}}(y)]$. We do so by bounding successively $\mathbb P_p[\mathcal A_1(y)]$ and $\mathbb P_p[\mathcal A_2(y)]$.

 Since $x_{j+1}-x_j$ must be of the form $2k\mathbf{e}_i$ for some $1\leq i \leq d$, if $s>k$, there are of order $s/k\asymp s/L(p)$ boxes $\Lambda_k(x_j)$ in $\Lambda_s(y)$. Hence,
\begin{equation}
	|\Lambda_s(y)\cap Q|\asymp \frac{s}{L(p)}\cdot L(p)^d.
\end{equation}
If $s<k$, $|\Lambda_s(y)\cap Q|\asymp s^d$. Combining these observations with Theorem \ref{thm: asselah schapira}, we obtain $c_4=c_4(d)>0$ such that, for $p$ sufficiently close to $p_c$ and every $y \in \Lambda_s(y)$,
\begin{equation}
	\mathbb P_p\Big[|\mathcal C\big(z; \Lambda_t(y)\cap Q\big)\cap \Lambda_s(y)|> s^4 (\log s)^4\Big]\leq \exp\Big(-c_4\Big[\frac{s}{L(p)}\vee 1\Big]^{4(1-1/d)}(\log s)^4\Big).
	\end{equation}
	In particular, if $p$ is sufficiently close to $p_c$,
	\begin{equation}\label{eq: bound A1(y)}
		\mathbb P_p[\mathcal A_1(y)^c]\leq \exp\Big(-\frac{c_4}{2}\Big[\frac{s}{L(p)}\vee 1\Big]^{4(1-1/d)}(\log s)^4\Big).
	\end{equation}
	
Next, note that any path from $Q\cap \partial \Lambda_s(y)$ to $Q\cap \partial \Lambda_t(y)$ must cross order $(t-s)/k \asymp (t-s)/L(p)$ boxes. Letting $z\in \partial \Lambda_s(y)\cap Q$, the upper bound of Theorem \ref{thm:nc bounds phi(S)} gives $c_5=c_5(d)>0$ such that, for $p$ sufficiently close to $p_c$,
\begin{equation}
	\mathbb P_p[z\connect{Q\:}\partial \Lambda_t(y)]\leq \sum_{w\in \partial \Lambda_t(y)}\tau_{\Lambda_t(y),p}(z,w)\leq p^{-1}\varphi_p(\Lambda_t(y)-z)\leq \exp\Big(-c_5\frac{t-s}{L(p)}\Big).
	\end{equation} 
	Observing that $(t-s)/L(p)\geq (\log s)^2$, we additionally obtain, for $p$ even closer to $p_c$,
	\begin{equation}\label{eq: pre bound A2(y)}
		\mathbb P_p[\partial \Lambda_s(y)\connect{Q\:}\partial \Lambda_t(y)]\leq \exp \Big(-\frac{c_5}{2}\frac{t-s}{L(p)}\Big).
	\end{equation}
Together with \eqref{eq:BK ineq}, \eqref{eq: pre bound A2(y)} yields, for $p$ sufficiently close to $p_c$,
\begin{equation}\label{eq: bound A2(y)}
	\mathbb P_p[\mathcal A_2(y)^c]\leq \exp\Big(-\frac{c_5}{2}\frac{t-s}{L(p)} \cdot (\log s)^3\Big). 
\end{equation}
Combining \eqref{eq: bound A1(y)} and \eqref{eq: bound A2(y)} gives the existence of $c_6=c_6(d)>0$ such that, for every $s>k^{1/(2d)}$ and every $p$ sufficiently close to $p_c$,
\begin{align}\label{eq: pf averaged 2}
\begin{split}
\mathbb P_p\big[\widetilde {\mathcal T}_{s,Q}^{\textrm{loc}}(y)^c\big] & \le \exp\Big(-c_6\Big[\frac{s}{L(p)}\vee 1\Big]^{4(1-1/d)}(\log s)^4\Big)+\exp\Big(-c_6 \frac{t-s}{L(p)} \cdot (\log s)^3\Big)  
\\
& \le 2 \exp\Big(-c_6 \frac{t-s}{L(p)}\cdot (\log s)^2\Big),
\end{split}		
\end{align}
where in the second inequality we used that (by definition of $t$), $[\tfrac{s}{L(p)}\vee 1]^{4(1-1/d)}(\log s)^4\geq \tfrac{t-s}{L(p)}\cdot (\log s)^2$.

We now define 
\begin{equation}
s_0 = \inf\big\{s\ge  k^{\frac 1{2d}}: \exists y \in F' \text{ such that } 0 \in \Lambda_t(y)\big\}, 
\end{equation}
and let $t_0$ be the real number associated to $s_0$ by~\eqref{def:st}.  By the induction hypothesis, there is $\nu=\nu(\alpha)>0$, such that 
\begin{equation}\label{eq: pf averaged 3}
\mathbb E_p\big[|\mathcal X_{Q'}\cap F'|\big]=\sum_{y\in F'} \tau_{Q',p}(0,y) \ge \exp\Big(-\nu \frac{t_0}{L(p)}\Big). 
\end{equation}
Combining \eqref{eq: pf averaged 2} and \eqref{eq: pf averaged 3}, and using that $t_0-s_0 \ge t_0/2$, gives, if $p$ is sufficiently close to $p_c$,
 \begin{equation}\label{eq: pf averaged 4}
\sum_{y\in F'}\sum_{s\geq s_0}\mathbb P_p\big[\widetilde {\mathcal T}_{s,Q}^{\textrm{loc}}(y)^c\big] \le \frac 14  \sum_{y\in F'} \tau_{Q',p}(0,y).
\end{equation}
Plugging \eqref{eq: pf averaged 1.5} and \eqref{eq: pf averaged 4} in \eqref{eq: pf averaged 1}, and using the trivial inclusion $\{0\connect{Q'\:}y\}\subset\{0\connect{Q'\:}\Lambda_t(y)\}$ yields, for $K$ large enough and for $p$ sufficiently close to $p_c$,
\begin{equation}\label{eq: pf averaged 5}
	\mathbb E_p\big[|\mathcal X_{Q',Q}^{K\textup{-reg}}\cap F'|\big]\geq \frac{1}{2}\sum_{y\in F'}\tau_{Q',p}(0,y)-\sum_{s_0>s>k^{1/(2d)}} \sum_{y\in F'} \mathbb P_p[0\connect{Q'}\Lambda_t(y), \widetilde {\mathcal T}_{s,Q}^{\textrm{loc}}(y)^c]. 
\end{equation}
If $s< s_0$, the origin is not part of $\Lambda_t(y)$: the events $\{0\connect{Q'\:} \Lambda_t(y)\}$ and $\widetilde{\mathcal T}_{s,Q}^{\textup{loc}}(y)$ are measurable with respect to disjoint sets of edges and are therefore independent. Plugging this observation and \eqref{eq: pf averaged 2} in \eqref{eq: pf averaged 5} gives\begin{multline}\label{cor:facets.final}
\mathbb E_p\big[|\mathcal X_{Q',Q}^{K\text{-reg}} \cap F'|\big] \\  \ge \frac 12 \sum_{y \in F'} \tau_{Q',p}(0,y) - 2\sum_{s_0> s>k^{1/(2d)}} \sum_{y\in F'} \mathbb P_p[0\connect{Q'}\partial \Lambda_t(y)] \exp\Big(-  c_6\frac{t-s}{L(p)}\cdot (\log s)^2\Big).  
\end{multline}
To compare the two terms above, we lower bound $\tau_{Q',p}(0,y)$ (for a fixed $y\in F'$) in terms of $\mathbb P_p[0\connect{Q'\:}\partial \Lambda_t(y)]$. We proceed as follows. Reveal an open path from $0$ to $\partial \Lambda_t(y)$ in $Q'$. Such a path ends at a point $w_1\in \Lambda_k(x_{j_0})\cap \partial \Lambda_t(y)$ for some $j_0<N+1$. Observe that since $t\geq L(p)(\log s)^2$, one has $N+1-j_0\gtrsim (\log s)^2\geq 3$ if $p$ is sufficiently close to $p_c$. Let $Q'':=\cup_{j_0+2\le j\leq N}\Lambda_{k}(x_j)$. Let $F''$ denote the facet of $\Lambda_{k}(x_{j_0+1})$ that touches $\Lambda_k(x_{j_0+2})$. We first construct a path from $0$ to $x_{N}$ in $Q'$ by concatenating a path from $x_N$ to $F''$ in $Q''$ which ends at some point $w_2\in F''$, and then a path from $w_2$ to $w_1$ in $(\Lambda_k(x_{j_0})\cup \Lambda_k(x_{j_0+1}))\cap \Lambda_t(y)$. By Lemma \ref{lem:KN}, this gives $c_7=c_7(d)>0$ such that
\begin{equation}
	\tau_{Q',p}(0,x_{N})\geq \mathbb P_p[0\connect{Q'\:}\partial \Lambda_t(y)]\cdot \exp\big(-c_7(\log k)^2\big)\cdot \mathbb P_p[x_{N}\connect{Q''\:}F''].
\end{equation}
However, by the induction hypothesis,
\begin{equation}
	\mathbb P_p[x_{N}\connect{Q''\:}F'']\geq \frac{1}{|F''|}\sum_{w\in F''}\tau_{Q'',p}(x_{N},w)\geq \exp\Big(-c_8(\log k)-\nu\frac{t}{L(p)}\Big),
\end{equation}
Putting all the pieces together, and recalling that $t/L(p)\geq (\log s)^2\gtrsim (\log k)^2$ gives, for $p$ sufficiently close to $p_c$,
\begin{equation}
	\tau_{Q',p}(0,x_{N})\geq \mathbb P_p[0\connect{Q'\:}\partial \Lambda_t(y)]\cdot \exp\Big(-\frac{\nu}{2}\frac{t}{L(p)}\Big).
\end{equation}
Now, by the FKG inequality \eqref{eq:FKG} and Lemma \ref{lem:KN}, for $p$ sufficiently close to $p_c$,
\begin{equation}\label{eq: pf averaged 6}
	\tau_{Q',p}(0,y)\geq \tau_{Q',p}(0,x_{N})\cdot \tau_{\Lambda_{k}(x_{N}),p}(x_{N},y)\geq \mathbb P_p[0\connect{Q'\:}\partial \Lambda_t(y)]\cdot \exp\Big(-\frac{\nu}{4}\frac{t}{L(p)}\Big).
\end{equation}
Plugging \eqref{eq: pf averaged 6} in~\eqref{cor:facets.final} and using that $t-s\ge t/2$ gives, for $K$ large enough and $p$ sufficiently close to $p_c$, 
\begin{equation}\label{eq: pf averaged 7}
\mathbb E_p\big[|\mathcal X_{Q',Q}^{K-\text{reg}} \cap F'|\big] \ge \frac 14 \sum_{y \in F'} \tau_{Q',p}(0,y). 
\end{equation}
Combining \eqref{eq: pf averaged 7} with~\eqref{cor:lowerbound1} and~\eqref{cor:lowerbound2}, and picking $\alpha:=\tfrac{c_2c_3}{4}\wedge c$ concludes the proof of the induction step. 
\end{proof}

\subsection{Pointwise near-critical lower bounds}
We finally turn to the proof of the lower bound in Theorem \ref{thm:bounds 2pt}. For the proof, we will need the following lemma. 
\begin{Lem}\label{lem:nearcritic2pt}
There exists $\kappa=\kappa(d)>0$ such that, for every $p\in [p_c/2,p_c)$, and every $x\in \mathbb Z^d$ satisfying $|x|\leq \kappa L(p)$, 
\begin{equation}
\frac 12 \tau_{p_c}(0,x) \le \tau_p(0,x) \le \tau_{p_c}(0,x). 
\end{equation}
\end{Lem}
\begin{proof} Let $\kappa>0$ to be fixed later.
The upper bound holds in fact for all $x\in \mathbb Z^d$, and follows by monotonicity of $\tau_p$ in $p$. For the lower bound, we combine Russo's formula (see \cite{GrimmettPercolation1999}), the BK inequality \eqref{eq:BK ineq}, and~\eqref{eq: 2pt full space estimate}, to obtain, for every $p\leq p_c$ and every $x\in \mathbb Z^d$, 
\begin{align}
\nonumber \frac{\mathrm{d}}{\mathrm{d}p}\tau_p(0,x) & \le \sum_{u\sim v} \mathbb P_p[\{0\connect{}u\}\circ \{v\connect{}x\}] \le \sum_{u\sim v} \tau_p(0,u)\tau_p(v,x) \lesssim \frac{1}{(1\vee |x|)^{d-4}} \\
& \lesssim (1+|x|)^2 \cdot \tau_{p_c}(0,x). 
\end{align}
Integrating the above inequality over the interval $[p,p_c]$ shows that, for every $|x|\le \kappa \cdot L(p)$, 
\begin{equation}\label{tauppc}
\tau_{p_c}(0,x) - \tau_p(0,x) \lesssim (p_c-p)(\kappa L(p))^2 \tau_{p_c}(0,x)\leq \frac{1}{2}\tau_{p_c}(0,x),
\end{equation}
where the last inequality follows from Theorem \ref{thm:sharplength} and by choosing $\kappa>0$ sufficiently small. Rearranging~\eqref{tauppc} concludes the proof of the lower bound. 
\end{proof}

We can now conclude the proof of Theorem \ref{thm:bounds 2pt}. 

\begin{proof}[Proof of the lower bound in Theorem~\textup{\ref{thm:bounds 2pt}}] 
Let $p\in [p_0,p_c)$ where $p_0<p_c$ will be fixed below. Let $\kappa>0$ be given by Lemma \ref{lem:nearcritic2pt}.

 If $|x|\le \kappa L(p)$, the result follows from Lemma \ref{lem:nearcritic2pt} together with \eqref{eq: 2pt full space estimate} and Theorem~\ref{thm:sharplength}. 
 
 Consider now $x=(x_1,\ldots,x_d)$ such that $|x|> \kappa L(p)$. By symmetry, we may assume that all the coordinates of $x$ are positive. Let $k := \lfloor \varepsilon L(p)\rfloor$, where $\varepsilon:=\tfrac{\kappa}{5}$. Take $p_0$ sufficiently close to $p_c$ so that $k\geq 1$. Consider a minimal (in length) sequence $0=x_0,\dots,x_N \in 2k\mathbb Z^d$, such that for any $j< N$, $x_{j+1} - x_j = 2k \mathbf e_i$, for some $i\in \{1,\dots,d\}$, and such that 
\begin{equation}
\mathrm{d}\big(x,\cup_{j\le N} \Lambda_k(x_j)\big)=  \mathrm{d}\big(x,\Lambda_k(x_N)\big) \in [k, 3k]. 
\end{equation}
Note that by construction $N\le d\tfrac{|x|}{k}$. Let $F$ be the facet of the cube $\Lambda_k(x_N)$ that is the closest to $x$, and set $Q := \cup_{j=0}^N \Lambda_k(x_j)$. By construction again, one has 
\begin{equation}
|x-u| \le 5 k \le \kappa L(p), \qquad \forall u\in F. 
\end{equation} 
Furthermore, by Proposition~\ref{prop: nc lb averaged}, one has $\alpha=\alpha(\kappa,d)>0$ such that, 
\begin{equation}\label{eq:nbrpioneer}
\sum_{y\in F} \tau_{Q,p}(0,y) \ge \alpha^N. 
\end{equation}
Observe that \eqref{eq:nbrpioneer} is enough to conclude if $|x|\ge L(p) (\log L(p))^2$. Indeed, combining it with Lemma \ref{lem:KN} and the fact that $N\lesssim |x|/L(p)$, gives, for some $c_1,\nu>0$ (which only depend on $d$) and for $p$ sufficiently close to $p_c$,  
\begin{align}\notag
\tau_p(0,x) &\ge \mathbb P_p[0\connect{Q\:}F] \cdot \inf_{y\in F}\tau_{\Lambda_{\mathrm{d}(x,F)},p}(y,x) \\&\ge  \Big(\frac 1{|F|}\sum_{y\in F} \tau_{Q,p} (0,y)\Big)\cdot \inf_{y\in F}\tau_{\Lambda_{\mathrm{d}(x,F)},p}(y,x) \notag
\\&\ge \alpha^N \exp\big(-c_1(\log L(p))^2\big)\notag
\\&\ge (1\vee |x|)^{2-d}\exp\Big(-\nu \frac{ |x|}{L(p)}\Big), 
\end{align}
which proves the result, thanks to Theorem~\ref{thm:sharplength}. 

We now assume that $|x|\le L(p) (\log L(p))^2$. We use a similar strategy as in the proof of Proposition \ref{prop: nc lb averaged}, except that since we are seeking a pointwise estimate, we need to add further restrictions to our set $\mathcal X_Q^K$: we impose a condition of local regularity around $x$. More specifically, for $K>0$, we let 
\begin{equation}
\mathcal X_Q^K(x) := \Big\{u\in  \mathcal X_Q^K : \mathcal C(u) \cap \Lambda_K(x) = \emptyset, |\mathcal C(u) \cap \Lambda_s(x)|\le 20 s^4 (\log s)^7, \ \forall s\ge K\Big\}. 
\end{equation}
Given $u\in F$, we say that a set $A$ is $u$-good if $\mathbb P_p[\mathcal C(u) = A]>0$ and if it is such that $u\in \mathcal X_Q^K(x)$. Additionally, as in the proof of Proposition \ref{prop: nc lb averaged}, we ask that for $u\in \mathcal X_Q^K(x)\cap F$, and writing $F=x_N+F^+_i(k)$ for some $1\leq i \leq d$, the $\mathbf{w}$ given by the $K$-escapability of $u$ satisfies $\mathbf{w}=u+\lfloor \tfrac{K}{2}\rfloor\mathbf{e}_i$. A minor modification of Proposition \ref{prop: effective reversed SL} (in which $\mathcal X_Q^K$ is replaced by $\mathcal X_Q^K(x)$) gives $c_2=c_2(K,d)>0$ such that
\begin{equation}\label{eq: pf pointwise 1}
	\tau_p(0,x)\geq c_2\sum_{u\in F}\mathbb E_p\Big[\mathds{1}\{u\in \mathcal X_Q^K(x)\}\Big(\tau_p(\mathbf{w},x)-\mathsf{E}(\mathbf{w},x;\mathcal C(u),p)\Big)\Big]
\end{equation}
 If $p$ is sufficiently close to $p_c$, so that $k\geq 2K$, one also has that $|\mathbf{w}-x|\leq \kappa L(p)$. By Lemma \ref{lem:nearcritic2pt}, this gives that $\mathbf{w}$ as above satisfies
 \begin{equation}
 	\tau_p(\mathbf{w},x)\gtrsim \frac{1}{L(p)^{d-2}}.
 \end{equation}
 Moreover, if $u\in F$ and $A$ is $u$-good, then \eqref{eq: 2pt full space estimate} yields
 \begin{equation}\label{eq: pf pointwise 2}
 	\mathsf{E}(\mathbf{w},x;A,p)=\sum_{a\in A}\tau_{p}(\mathbf{w},a)\tau_p(a,x)\lesssim \frac{1}{|x|^{d-2}}\Big(\sum_{a\in A\setminus \Lambda_{|x|/2}(x)}\tau_p(\mathbf{w},a)+\sum_{a\in A\setminus \Lambda_{|x|/2}(\mathbf{w})}\tau_p(a,x)\Big).
 \end{equation}
 The same computation as in \eqref{eq:pf main rep5} gives that
 \begin{equation}\label{eq: pf pointwise 3}
 	\sum_{a\in A}\tau_p(\mathbf{w},a)\lesssim \frac{1}{\sqrt{K}}.
 \end{equation}
 Now, the additional restriction imposed on $\mathcal X_Q^K(x)$ compared to $\mathcal X_Q^K$ allows one to bound the second sum in \eqref{eq: pf pointwise 2} as follows:
 \begin{equation}\label{eq: pf pointwise 4}
 	\sum_{a\in A}\tau_p(a,x)\lesssim \sum_{\ell\geq \log_2(K)}\frac{|A\cap \Lambda_{2^\ell}(x)|}{2^{\ell(d-2)}}\leq \sum_{\ell\geq \log_2(K)}\frac{20(2^\ell)^4\ell^7}{2^{\ell(d-2)}}\lesssim \frac{1}{\sqrt{K}}.
 \end{equation}
 Plugging \eqref{eq: pf pointwise 3} and \eqref{eq: pf pointwise 4} into \eqref{eq: pf pointwise 2} gives, for $K$ large enough,
 \begin{equation}
 	\tau_p(\mathbf{w},x)-\mathsf{E}(\mathbf{w},x;\mathcal C(u),p)\gtrsim \frac{1}{L(p)^{d-2}}.
 \end{equation}
 Plugging the above displayed equation in \eqref{eq: pf pointwise 1} then yields, for $K$ large enough, the existence of $c_3=c_3(K,d)>0$ such that
 \begin{equation}\label{eq: pf pointwise 4.5}
 	\tau_{p}(0,x)\geq \frac{c_3}{L(p)^{d-2}}\cdot\mathbb E_p\big[|\mathcal X_Q^K(x)\cap F|\big].
 \end{equation}
 It remains to control the expectation above. As in the proof of Proposition \ref{prop: nc lb averaged}, let $\mathcal X_Q^{K\text{-reg}}$ denote  the set of $K$-regular pioneers of $Q$, and define
\begin{equation}
\mathcal X_Q^{K\text{-reg}}(x) := \Big\{y\in  \mathcal X_Q^{K\text{-reg}} : |\mathcal C(y) \cap B(x,s)|\le s^4 (\log s)^7, \ \forall s\ge K\Big\}. 
\end{equation}
The same argument as in Lemma \ref{lem: regular are in average also escapable } gives that for $K$ potentially even larger, there is $c_4=c_4(K,d)>0$ such that
\begin{equation}\label{2pt.smallx2}
\mathbb E_p\big[|\mathcal X_Q^K(x)\cap F|\big]\ge c_4\cdot\mathbb E_p\big[|\mathcal X_Q^{K\text{-reg}}(x)\cap F|\big]. 
\end{equation} 
Recall the definition of the events $\mathcal T_s^{\text{loc}}(x)$ from~\eqref{eq:def ts loc}, and let 
\begin{equation}
\mathcal E_K(x) := \bigcap_{K\le s \le k^{1/(2d)}} \mathcal T_s^{\text{loc}}(x). 
\end{equation}
Now, 
\begin{multline}\label{eq: pf pointwise 5}
	\mathbb E_p\big[|\mathcal X_Q^{K\text{-reg}}(x)\cap F|\big]\geq \sum_{y\in F}\tau_{Q,p}(0,y)-\sum_{y\in F}\mathbb P_p[0\connect{Q\:}y, \: \mathcal E_K(x)^c]\\-\sum_{s:\: K\leq s\leq k^{1/(2d)}}\sum_{y\in F}\mathbb P_p[0\connect{Q\:}y, \: \mathcal T_s^{\textup{loc}}(y)^c]-\sum_{s> k^{1/(2d)}}\sum_{y\in F}\Big(\mathbb P_p[\mathcal T_s^{\textup{loc}}(x)^c]+\mathbb P_p[\mathcal T_s^{\textup{loc}}(y)^c]\Big)
\end{multline}
Note that $\mathcal E_K(x)$ is independent of the configuration of edges inside $Q$, and thus 
\begin{equation}
\sum_{y \in F} \mathbb P_p[0\connect{Q\:}y, \mathcal E_K(x)^c] = \mathbb P_p[\mathcal E_K(x)^c]\cdot \sum_{y\in F} \tau_{Q,p}(0,y).
\end{equation}
By Lemma \ref{lem:Tsloc}, if $K$ is large enough, $\mathbb P_p[\mathcal E_K(x)]\geq \tfrac{3}{4}$.

Moreover, as in \eqref{eq: pf averaged 1.5}, if $K$ is large enough
\begin{equation}
	\sum_{s:\: K\leq s\leq k^{1/(2d)}}\sum_{y\in F}\mathbb P_p[0\connect{Q\:}y, \: \mathcal T_s^{\textup{loc}}(y)^c]\leq \frac{1}{4}\sum_{y\in F}\tau_{Q,p}(0,y).
\end{equation}
Another use of Lemma \ref{lem:Tsloc} gives $c_5,c_6>0$ (which only depend on $d$) such that, if $p$ is sufficiently close to $p_c$,
\begin{equation}
	\sum_{s> k^{1/(2d)}}\sum_{y\in F}\Big(\mathbb P_p[\mathcal T_s^{\textup{loc}}(x)^c]+\mathbb P_p[\mathcal T_s^{\textup{loc}}(y)^c]\Big)\leq \sum_{s>k^{1/(2d)}}\exp\Big(-c_5(\log s)^4\Big)\leq \exp\Big(-c_6(\log k)^4\Big).
\end{equation}
Plugging the three last displayed equations in \eqref{eq: pf pointwise 5} gives, for $K$ large enough, and $p$ sufficiently close to $p_c$,
\begin{equation}
	\mathbb E_p\big[|\mathcal X_Q^{K\text{-reg}}(x)\cap F|\big]\geq \frac{1}{2}\sum_{y\in F}\tau_{Q,p}(0,y)-\exp\Big(-c_6(\log k)^4\Big).
\end{equation}
Using~\eqref{eq:nbrpioneer} and the fact that $|x|\le L(p) (\log L(p))^2$, we deduce that for $K$ large enough and $p$ sufficiently close to $p_c$, 
\begin{equation}
\mathbb E_p\big[|\mathcal X_Q^{K\text{-reg}}(x)\cap F|\big] \gtrsim \alpha^N. 
\end{equation}
Combining this with~\eqref{eq: pf pointwise 4.5}, \eqref{2pt.smallx2}, and the fact that $N\lesssim |x|/L(p)$ concludes the proof. 
\end{proof}

\bibliographystyle{plain}
\bibliography{biblio.bib}
\end{document}